\documentclass[a4paper,11pt]{amsart}

\usepackage{amssymb}
\usepackage{amsfonts}
\usepackage{amsmath}

\usepackage{graphicx}
\usepackage[all]{xy}
\usepackage{array}

\usepackage{amsthm}

\newtheorem{theorem}{Theorem}[section]
\newtheorem{lemma}[theorem]{Lemma}
\newtheorem{proposition}[theorem]{Proposition}
\newtheorem{corollary}[theorem]{Corollary}
\newtheorem{question}[theorem]{Question}
\newtheorem{problem}[theorem]{Problem}
\newtheorem{conjecture}[theorem]{Conjecture}

\theoremstyle{definition}
\newtheorem{definition}[theorem]{Definition}

\newtheorem{remark}[theorem]{Remark}
\newtheorem{example}[theorem]{Example}

\newcommand{\Si}{\mathfrak{S}}
\newcommand{\Ext}{{\mathrm{Ext}}}
\renewcommand{\hom}{{\mathrm{Hom}}}
\newcommand{\End}{{\mathrm{End}}}
\newcommand{\Id}{{\mathrm{Id}}}
\renewcommand{\k}{\Bbbk}
\newcommand{\op}{\mathrm{op}}
\renewcommand{\P}{\mathcal{P}}
\newcommand{\V}{\mathcal{V}}
\newcommand{\Z}{\mathbb{Z}}

\input cyracc.def
\font\tencyr=wncyr10
\def\cyr{\tencyr\cyracc}
\newcommand{\sha}{{\mbox{\cyr Sh}}}

\title{Troesch complexes and extensions of strict polynomial functors}
\author{Antoine Touz\'e }
\thanks{The author was partially supported by the ANR HGRT (Projet BLAN08-2 338236): \emph{Nouveaux liens entre la th\'eorie de l'homotopie et la th\'eorie des groupes et des repr\'esentations}} 
\date{\today}
\begin{document}

\sloppy

\begin{abstract}
We develop a new approach of extension calculus in the category of strict polynomial functors, based on Troesch complexes. We obtain new short elementary proofs of numerous classical $\Ext$-computations as well as new results. 

In particular, we get a cohomological version of the `fundamental theorems' from classical invariant invariant theory for $GL_n$ for $n$ big enough (and we give a conjecture for smaller values of $n$).
 
We also study the `twisting spectral sequence' $E^{s,t}(F,G,r)$
converging to the extension groups $\Ext^*_{\P_\k}(F^{(r)}, G^{(r)})$ between the twisted functors $F^{(r)}$ and $G^{(r)}$. 
Many classical $\Ext$ computations simply amount to the collapsing of this spectral sequence at the second page (for lacunary reasons), and it is also a convenient tool to study the effect of the Frobenius twist on $\Ext$ groups. 
We prove many cases of collapsing, and we conjecture collapsing is a general fact. 

\end{abstract}

\maketitle


\section{Introduction}

Let $\k$ be a field of prime characteristic $p$. In \cite{FS}, Friedlander and Suslin introduced the category $\P_\k$ of strict polynomial functors (of finite type)  over $\k$. 
Let $\mathcal{V}_\k$ the category of finite dimensional $\k$-vector spaces. Roughly speaking, objects of $\P_\k$ are functors $F:\mathcal{V}_\k\to \mathcal{V}_\k$ with some additional polynomial structure, so that the following property holds. If $G$ is an algebraic group (or group scheme) over $\k$ acting rationally on $V$, then functoriality defines a \emph{rational} action of $G$ on $F(V)$. Such functors occur very frequently in representation theory. Typical examples are symmetric powers $S^n$, tensor powers $\otimes^d$, exterior powers $\Lambda^n$ or divided powers $\Gamma^n$.

The category $\P_\k$ is particularly suited to study the representation theory of $GL_n$. Indeed, 
evaluation on the standard representation $\k^n$ of $GL_n$ induces a map:
$$\Ext^*_{\P_\k}(F,G)\to \Ext^*_{GL_n}(F(\k^n),G(\k^n))\;,$$
which is an isomorphism as soon as $n\ge \max\{\deg F,\deg G\}$ \cite[Cor 3.13]{FS}. 
Thus, one can use the powerful computational tools available in $\P_\k$ to compute the `stable' (that is, when $n$ is big enough) extension groups between rational $GL_n$-modules.

A successful application of strict polynomial functors is the computation of extension groups between representations involving `Frobenius twists'.
If $V$ is a rational representation of $GL_n$, we denote by $V^{(r)}$ the rational representation of $GL_n$ obtained by twisting along the $r$-th power of the Frobenius morphism. The functor $I^{(r)}:V\mapsto V^{(r)}$ is actually a strict polynomial functor. It has a crucial role in many problems, for example in cohomological finite generation problems \cite{FS,VdKGross,TVdK}. Twisted representations are also related \cite{CPSVdK} to the cohomology of the finite groups $GL_n(\mathbb{F}_q)$. 
For these reasons, extension groups between twisted functors (that is, extension groups of the form $\Ext^*_{\P_\k}(F\circ I^{(r)},G\circ I^{(r)})$)  have received much attention, and many successful computations have been performed \cite{FS,FFSS,Chalupnik1,Chalupnik2}. In particular, Cha{\l}upnik has proved in \cite[Thm 4.3]{Chalupnik1} that $\Ext$-groups of the form $\Ext^*_{\P_\k}(\Gamma^\mu\circ I^{(r)},F\circ I^{(r)})$ (where $\Gamma^\mu$ denotes a tensor product of divided powers) can be easily computed via an isomorphism:
$$\Ext^*_{\P_\k}(\Gamma^\mu\circ I^{(r)},F\circ I^{(r)})\simeq \hom_{\P_\k}(\Gamma^\mu\circ (E_r\otimes I),F)\quad (*)$$
where $E_r$ is the graded vector space $\Ext^*_{\P_\k}(I^{(r)},I^{(r)})$, and $E_r\otimes I$ denotes the graded functor $V\mapsto E_r\otimes V$.

\bigskip

In this article, we give a new approach of $\Ext$-computations between twisted functors. This approach does not depend on the earlier computations of \cite{FS,FFSS,Chalupnik1,Chalupnik2}.
As main tool, we use the explicit injective coresolutions of twisted symmetric powers built by Troesch in \cite{Troesch}. These coresolutions generalize to all prime characteristic what was previously known in characteristic $p=2$ only \cite{FLS,FS}. At first sight these coresolutions are quite big and complicated (especially in odd characteristic), and one could fear that they are useless for concrete computations. But this is not the case: we show that only a very little part of the information contained in these coresolutions is needed for computations. In particular, we don't need the information borne by the differential (see lemma \ref{lm-CFSr})! We exploit the latter fact to get the first main result of the article, namely:
\begin{itemize}
\item we get a new and simpler proof of Cha{\l}upnik's isomorphism $(*)$, and we derive from this isomorphism new short proofs of many $\Ext$-computations.
\end{itemize}
Then we try to go further in the study of extension groups between twisted functors. With the help of Troesch complexes, we obtain new results in two independent directions.
\begin{itemize}
\item First, we apply isomorphism $(*)$ to compute rational cohomology algebras of $GL_n$. To do this, we need to improve Friedlander and Suslin's bound on $n$ so that $\Ext^*_{\P_\k}(\Gamma^\mu\circ I^{(r)},F\circ I^{(r)})$ computes $GL_n$ extensions. As an application, we give explicit generators and relations for the cohomology algebra $H^*(GL_n,A)$, where $A$ is an algebra of polynomials over a direct sum of copies of the twisted standard representation $(\k^{n})^{(r)}$ and its dual (for $n$ big enough) in the spirit of classical invariant theory. We give a conjecture for smaller $n$.
\item We introduce in section \ref{sec-untwist} the `twisting spectral sequence' which generalizes isomorphism $(*)$. We show that this spectral sequence contains interesting information about extensions between twisted functors. As main new result, we prove in section \ref{sec-collapse} that in many cases this spectral sequence collapses at the second page, and we conjecture that this is a general fact.  A positive answer to this conjecture would improve significantly our understanding of the homological effects of Frobenius twists. 
\end{itemize}

\bigskip

Let us review more specifically the content of the paper. Sections \ref{sec-background} and \ref{sec-Troesch} are mainly expository. They collect well-know facts 
about $\P_\k$ and describe the properties of Troesch coresolutions which we need for our computations (in propositions \ref{prop-res1} and \ref{prop-res-2}). For our computations, we do not need an explicit description of the differentials of Troesch coresolutions. However, for the reader's convenience, we have recalled their construction in an appendix.

In section \ref{sec-classical}, we present our elementary proof of Cha{\l}upnik's isomorphism $(*)$. As corollaries, we retrieve many computations from \cite{FS,FFSS,Chalupnik1}. Beside brevity, our proofs have the advantage to avoid the use of many technical tools (e.g. functors with several variables, generalized Koszul complexes, trigraded Hopf algebra structures on hypercohomology spectral sequences, symmetrizations of functors) which seemed essential in \cite{FS,FFSS}, and also in \cite{Chalupnik1} since this latter article elaborates on the results of \cite{FFSS}.

In section \ref{subsec-mult}, we enrich the computation of $\Ext^*_{\P_\k}(\Gamma^\mu\circ I^{(r)},F\circ I^{(r)})$ by describing cup products as well as the twisting map:
$$\mathrm{Fr}_1:\Ext^*_{\P_\k}(\Gamma^\mu\circ I^{(r)},F\circ I^{(r)})\to \Ext^*_{\P_\k}(\Gamma^\mu\circ I^{(r+1)},F\circ I^{(r+1)})$$ induced by precomposition by $I^{(1)}$. Although the result might not surprise experts, it is neither stated, nor proved in the literature. We use it to generalize some Hopf algebra computations of \cite{FFSS}.

In section \ref{sec-semistable}, we try to apply our $\Ext$-computations in $\P_\k$ to compute some rational cohomology algebras for $GL_n$. When doing so, we encounter the problem that Friedlander and Suslin's bound on $n$ such that $\Ext^*_{\P_\k}(F\circ I^{(r)},G\circ I^{(r)})$ computes $GL_n$ extensions is not sufficient. Fortunately, with the help of Troesch complexes, we prove that this bound can be substantially improved for extensions of the form $\Ext^*_{\P_\k}(\Gamma^\mu\circ I^{(r)},F\circ I^{(r)})$.  Combining this with the previous computations of section \ref{subsec-mult} we prove a cohomological analogue of the  `Fundamental Theorems' from classical invariant theory. Namely, we describe the cohomology algebra $H^*(GL_n, S^*((\k^{n\,(r)})^{\oplus k}\oplus (\k^{n\,(r)})^{\vee\,\oplus \ell})$. Actually our result is valid for $n\ge p^r\min\{k,\ell\}$. For $n$ smaller, we state a conjecture. 

We introduce the `twisting spectral sequence' $E^{s,t}(F,G,r)$ in section \ref{sec-untwist}. The second page of this spectral sequence is given by extension groups between $F$ and $G$ precomposed by the functor $V\mapsto E_r\otimes V$, and it converges to the extension groups $\Ext^*_{\P_\k}(F\circ I^{(r)}, G\circ I^{(r)})$. Although the construction of the twisting spectral sequence is a formal consequence of Cha{\l}upnik's isomorphism, it is an interesting tool to study extensions between twisted functors. Indeed, the effect of cup products and the effect of the twisting map $\mathrm{Fr}_1$ may be easily read on the second page, so the twisting spectral sequence is a convenient way to study them. For example, we can read the `twist stability' phenomenon \cite[Cor 4.10]{FFSS} on the second page, and the injectivity of the twisting map $\mathrm{Fr}_1$ is implied by the collapsing at the second page. Also, many classical computations amount to the collapsing (for lacunary reasons) of the twisting spectral sequence at the second page. 

In fact, we observe that the twisting spectral sequence collapses at the second page in all the computations known so far, even when there is no lacunary reason for it, and we conjecture this is a general fact. As a main result, we make a step towards this conjecture, by proving in section \ref{sec-collapse} that $E(F,G,r)$ collapses at the second page for all $r\ge 0$ and many pairs $(F,G)$, including all pairs $(F,G)$ studied in \cite{FS,FFSS,Chalupnik1,Chalupnik2}. We also propose to the reader a combinatorial problem whose positive solution would prove the collapsing for any $F,G$.

\section{Background and notations}\label{sec-background}
In the article, we assume from the reader only a basic knowledge of the category $\P_\k$, corresponding to section 2 of the seminal article \cite{FS}. We do not assume that any $\Ext$-computation is known (we redo all the computations from scratch). In this background section,  we  introduce notations (most of them are standard), and we write down a few useful facts which are either implicit in, or easy consequences of \cite[Section 2]{FS}. 

\subsection{Notations}
Throughout the article, $\k$ is a field of prime characteristic $p$. If $V$ is a $\k$-vector space, we let $V^\vee:=\hom_\k(V,\k)$.   Many notations are as in \cite{FS}. In particular $F^\sharp(V):= F(V^\vee)^\vee$ denotes the dual of a functor $F$ (as in \cite[Prop 2.6]{FS}) and $F^{(r)}$ denotes the composite $F\circ I^{(r)}$. For the sake of simplicity, we drop the index `$\P_\k$' on $\hom$ and $\Ext$-groups when no confusion is possible (i.e. $\Ext^*(F,G)$ means $\Ext^*_{\P_\k}(F,G)$).

We denote tuples of nonnegative integers by Greek letters $\lambda,\mu,\nu$. Let $\mu=(\mu_1,\dots,\mu_n)$ be a tuple. The weight of $\mu$ is the integer $\sum \mu_i$. If $m$ is an integer, we denote by $m\mu$ the tuple $(m\mu_1,\dots,m\mu_n)$. We say that a positive integer $d$ divides $\mu$ if for all $1\le j\le n$, $d$ divides $\mu_j$. 
If $X$ is one of the symbols $S,\Lambda,\Gamma$, we denote by $X^\mu$ the tensor product $X^{\mu_1}\otimes\dots \otimes X^{\mu_n}$ (since $X^0=\k$, this tensor product has a meaning even if some $\mu_i$ are zero). 
 
We denote by $I$ the identity functor ($I=\Lambda^1=\Gamma^1=S^1$), and if $W$ is a finite dimensional vector space we denote by $F(W\otimes I)$ the precomposition of $F$ by the functor $V\mapsto W\otimes V$.

Finally, we denote by $\sha_r$ the finite dimensional graded vector space which equals $\k$ in degrees $i$ for $0\le i<p^r$ and which is zero in the other degrees (we use the Cyrillic letter `sha' by analogy with a dirac comb). We also denote by $E_r$ the even degree version of $\sha_r$, that is, $E_r$ is the graded vector space which equals $\k$ in degrees $2i$, for $0\le i< p^r$ and which is zero in the other degrees. We shall prove below in corollary \ref{cor-Ext-I-I} that $E_r$ is isomorphic to $\Ext^*(I^{(r)},I^{(r)})$, a result originally proved by Friedlander and Suslin and which inspires our notation.

\subsection{The exponential formula}\label{subsec-expformula}

Let $X$ denote one of the symbols $S,\Lambda,\Gamma$. For all $V,W$, there is an isomorphism (natural in $V,W$, and associative in the obvious sense)
$$ X^*(V)\otimes X^*(W)\simeq X^*(V\oplus W)\;.$$
According to \cite{FFSS}, we call this isomorphism the `exponential formula'. 
As a consequence of the exponential formula, $X^d(\bigoplus_{i=1}^n V_i)$ is isomorphic to $\bigoplus_\mu \bigotimes_{i=1}^n X^\mu_i(V_i)$ where the sum is taken over the $n$-tuples $\mu=(\mu_1,\dots,\mu_n)$ of nonnegative integers of weight $d$. 

\subsection{Maps between tensor products}

Let $d$ be a positive integer and let $\k\Si_d$ be the group algebra of the symmetric group. We define a $\k$-linear map
$$\k\Si_d\to \hom(\otimes^d, \otimes^d)\quad (*)$$
by sending a permutation $\tau\in\Si_n$ onto the natural transformation (still denoted $\tau$) which maps $v_1\otimes\dots\otimes v_d$ to $v_{\tau^{-1}(1)}\otimes\dots\otimes v_{\tau^{-1}(d)}$.  

\begin{lemma}\label{lm-descrHomTT}
The map $(*)$ is an isomorphism.
\end{lemma}
\begin{proof} Let $(b_1,\dots,b_d)$ be a basis of $\k^d$. For all $\tau\in\Si_d$, the natural transformation $\tau\in \hom(\otimes^d, \otimes^d)$ sends $b_1\otimes\cdots\otimes b_d\in (\k^d)^{\otimes d}$ onto $b_\tau=b_{\tau^{-1}(1)}\otimes\cdots\otimes b_{\tau^{-1}(d)}\in (\k^d)^{\otimes d}$. Since the family $(b_\tau)_{\tau\in\Si_d}$ is free, one gets that $(*)$ is injective.
The subspace of weight $(1,\dots,1)$ of the $GL_d$-representation $(\k^d)^{\otimes d}$ has basis $(b_\tau)_{\tau\in\Si_d}$. By \cite[Cor 2.12]{FS}, this subspace is isomorphic to $\hom(\otimes^d, \otimes^d)$. So for dimension reason, $(*)$ is an isomorphism.
\end{proof}

\subsection{Precomposition by Frobenius Twists}
Let $F,G$ be two strict polynomial functors. Precomposing extensions by $I^{(r)}$ yields a 
graded $\k$-linear map (natural in $F$,$G$): 
$$\mathrm{Fr_r}:\Ext^*(F,G)\to \Ext^*(F^{(r)},G^{(r)})\;.$$
The following lemma asserts that this twisting map is an isomorphism in degree zero (further properties of this map are proved in sections \ref{subsec-mult} and \ref{sec-untwist}).

\begin{lemma}\label{lm-twist}
Let $r$ be a nonnegative integer. 
For all tuple $\mu$ of nonnegative integers and all $F\in\P_\k$, precomposition by $I^{(r)}$ and the inclusion $S^{\mu\,(r)}\hookrightarrow S^{p^r\mu}$ induce isomorphisms:
\begin{align*}&\hom(F,S^\mu)\xrightarrow[]{\simeq} \hom(F^{(r)},S^{\mu\,(r)})\,,\\& \hom(F^{(r)},S^{\mu\,(r)})\xrightarrow[]{\simeq} \hom(F^{(r)},S^{p^r\mu})\,.
\end{align*}
As a result, for all $F$, $G$, precomposition by $I^{(r)}$ yields an isomorphism:
$$\hom(F,G)\xrightarrow[]{\simeq}  \hom(F^{(r)},G^{(r)})\;.$$
\end{lemma}
\begin{proof}
The $S^\mu$ form a (injective) cogenerator in $\P_\k$, so the last isomorphism of lemma \ref{lm-twist} follows from the first  one by taking coresolutions.

It remains to prove the first two isomorphisms. We prove them simultaneously. First, the map $\hom(F,S^\mu)\to \hom(F^{(r)},S^{\mu\,(r)})$, $f\mapsto f^{(r)}$ is injective. Indeed, if $f^{(r)}=0$ then using the isomorphism $V\simeq V^{(r)}$ (not natural in $V$, but which holds for dimension reasons) one gets that $f=0$.
The second map is induced by the injection $S^{\mu\,(r)}\hookrightarrow S^{p^r\mu}$, so it is also injective (by left exactness of $\hom(F,-)$). Now, to prove the isomorphisms, it suffices to prove that $\hom(F,S^\mu)$ and $\hom(F^{(r)},S^{p^r\mu})$ have the same dimension, or equivalently that $\hom(\Gamma^\mu,F^\sharp)$ and $\hom(\Gamma^{p^r\mu}, F^{\sharp\,(r)})$ have the same dimension. But we know their dimensions are equal by \cite[Cor 2.12]{FS}.
\end{proof}

As a consequence of lemma \ref{lm-twist}, one gets an elementary vanishing lemma:
\begin{lemma}\label{lm-vanish}
Let $r$ be a positive integer and  let $\mu$ be a tuple of nonnegative integers. If $p^r$ does not divide $\mu$, then $\hom(F^{(r)},S^{\mu})=0$.
\end{lemma}
\begin{proof}
By duality  \cite[Prop 2.6]{FS} and the Yoneda lemma \cite[Th 2.10]{FS}, the two vector spaces $\hom(F,S^{d}(\k^n\otimes I))$ and $\hom(F^{(r)},S^{p^rd}(\k^n\otimes I))$ have the same dimension (which equals the dimension of $F^\sharp(\k^n)$). Decomposing these vector spaces by the exponential formula and using the second isomorphism of lemma \ref{lm-twist}, we get the cancellation for the $n$-tuples $\mu$ of weight $d$ which are not divisible by $p^r$.
\end{proof}

\subsection{Values of strict polynomial functors on graded vector spaces}\label{subsec-ValuesGraded}

Let $\V_\k$ (resp. $\mathcal{V}_\k^*$) denotes the category of (resp. graded) finite dimensional $\k$-vector spaces and $\k$-linear maps (resp. which preserve the grading). 
Strict polynomial functors are functors $F:\V_\k\to\V_\k$ with a `strict polynomial structure'. This strict polynomial structure can be used to extend $F$ into a functor $F:\V_\k^*\to\V_\k^*$, as follows.

Let $V^*\in \V_\k^*$, and let $V\in\V_\k$ be the ungraded vector space obtained by forgetting the grading. We define $F(V^*)$ by:
\begin{itemize}
\item[(i)] As a vector space, $F(V^*)$ equals $F(V)$.
\item[(ii)] Let the multiplicative group $\mathbb{G}_m$ act on each $V^i$ with weight $i$. The strict polynomial structure of $F$ endows $F(V)$ with a rational action of $\mathbb{G}_m$. Thus $F(V)$ has a weight space decomposition $F(V)=\bigoplus_{i\ge 0} F(V)^i$. The grading on $F(V)$ is defined by putting $F(V)^i$ in degree $i$.
\end{itemize}
The following lemma is a trivial verification.
\begin{lemma}\label{lm-valuesgraded}
Let $V^*\in\V_\k^*$. Then evaluation on $V^*$ yields an exact functor
$\mathrm{ev}_{V^*}: \P_\k\to \V_\k^*$ which commutes with tensor products: $\mathrm{ev}_{V^*}(F\otimes G)=\mathrm{ev}_{V^*}(F)\otimes \mathrm{ev}_{V^*}(G)$.
Moreover, a graded map $V^*\to W^*$ induces a natural transformation $\mathrm{ev}_{V^*}\to \mathrm{ev}_{W^*}$. 
\end{lemma}

\begin{remark}\label{rk-2emedef}
Alternatively, the grading on $F(V)$ may be described in the following way.  Let $(v_i)_{1\le i\le n}$ be a homogeneous basis of $V^*$, and let $\mathbb{T}^{n}\simeq\mathbb{G}_m^{\times n}$ be the corresponding subtorus of $GL(V)$.
The $GL(V)$ representation $F(V)$ decomposes as a direct sum of $\mathbb{T}^n$-weight spaces: $F(V)= \bigoplus F(V)^{(d_1,\dots,d_n)}$. The grading on $F(V)$ is obtained by placing the summand $F(V)^{(d_1,\dots,d_n)}$ in degree $\sum d_i\deg(v_i)$ (to see it coincides with the former definition of $F(V^*)$, use the morphism $\mathbb{G}_m\to \mathbb{G}_m^{\times n}$, $x\mapsto (x^{\deg(v_i)})_{1\le i\le n}$).
\end{remark}

\begin{example}
For symmetric, exterior, divided or tensor powers, the notion of degree is nothing but the usual one. For example, let $V^*=\k e_1\oplus \k e_2$, with $e_i$ in degree $i$, and let $F=\Lambda^2\otimes S^3$. Then $e_1\wedge e_2\otimes e_1e_1e_2$ is an element of degree $1+2+1+1+2=7$ of $F(V^*)$.
\end{example}

\begin{example}\label{ex-twistgraded}
Let $V^*$ be a graded vector space and let $r$ be a nonnegative integer. Then $(V^*)^{(r)}$ is a graded vector space with $(V^i)^{(r)}$ in degree $ip^r$.
\end{example}

Let $F\in \P_\k$.  If $V^*$ is a graded vector space and $W$ is a vector space (concentrated in degree $0$), then the degree defined above splits the strict polynomial functor $F(V^*\otimes I):W\mapsto F(V^*\otimes W)$ into a direct sum where each summand takes values of homogeneous given degree. Thus we may consider $F(V^*\otimes I)$ as an element of the category $\P_\k^*$, whose objects are strict polynomial functors equipped with a grading, and whose maps are natural transformations which respect the grading.
From this viewpoint, lemma \ref{lm-valuesgraded} may be reformulated in the following way.
\begin{lemma}\label{lm-fctgraded}
Evaluation on $V^*\otimes I$ yields an exact functor:
$$\mathrm{ev}_{V^*\otimes I}: \P_\k\to \P_\k^*$$
which commutes with tensor products. Moreover, a graded map $V^*\to W^*$ induces a natural transformation $\mathrm{ev}_{V^*\otimes I}\to \mathrm{ev}_{W^*\otimes I}$. 
\end{lemma}

The following example is one of the basic examples of graded functors that we shall use in the paper.

\begin{example}\label{ex-sha}
Let $r,d$ be nonnegative integers. Recall that $\sha_r$ denotes the graded $\k$-vector space which is one dimensional in degrees $i$ for $0\le i<p^r$ and zero dimensional in the other degrees. Then by the exponential formula we have
$S^d(\sha_r\otimes I)\simeq \bigoplus S^\mu$
where the sum is taken over all the $p^r$-tuples $\mu=(\mu_0,\dots,\mu_{p^r-1})$ of nonnegative integers of weight $d$. With this latter description, the summand $S^\mu$ has degree $\sum i\mu_i$.
\end{example}

Finally we give a graded version of the Yoneda lemma \cite[Th 2.10]{FS}.
\begin{lemma}[The Yoneda lemma]\label{lm-Yoneda-graded}
Let $V^*$ be a graded finite dimensional vector space, and let $d$ be a nonnegative integer. 
Then for all $F\in\P_{\k,d}$ one has graded isomorphisms (natural in $F,V^*$):
\begin{align*}&\hom(\Gamma^d((V^*)^\vee\otimes I),F)\simeq F(V^*)\;,\\&  \hom(F,S^d(V^*\otimes I))\simeq F^\sharp(V^*)\;.
\end{align*}
\end{lemma}
\begin{proof}The second isomorphism is obtained from the first by duality, since $\hom(F,S^d(V^*\otimes I))$ is naturally isomorphic to $\hom(\Gamma^d((V^*)^\vee\otimes I),F^\sharp)$ by \cite[Prop 2.6]{FS}.
The first isomorphism is given by \cite[Th 2.10]{FS}. We have to prove that it respects the gradings. But this follows from \cite[Cor 2.12]{FS} and remark \ref{rk-2emedef}.
\end{proof}

\section{Troesch coresolutions}\label{sec-Troesch}

The main results of this section are propositions \ref{prop-res1} and \ref{prop-res-2}. These two propositions are blackboxes which we use as a key ingredient for almost all the theorems of the article. They give the properties of the  Troesch coresolutions $B_\mu(r)^*$. These coresolutions are injective coresolutions of twisted injectives of $\P_\k$ and they are derived from Troesch's main theorem in \cite{Troesch}. Since Troesch coresolutions are in fact $p$-coresolutions, we first recall some elementary facts about $p$-complexes.

\subsection{Recollections of p-complexes}
A $p$-complex in $\P_\k$ is a graded functor $C^*=\textstyle\bigoplus_{i\ge 0}C^i$ together with maps $d:C^i\to C^{i+1}$ such that $d^p=0$. 
For all $1\le s< p$, we can `contract' a $p$-complex $C^*$ into an ordinary complex $(C_{[s]})^*$ by taking alternatively $d^s$ and $d^{p-s}$ as differentials:
$$ (C_{[s]})^* = C^0 \xrightarrow[]{d^s} C^s \xrightarrow[]{d^{p-s}} C^p\xrightarrow[]{d^s} C^{p+s}\to \dots\;.$$

A $p$-complex $C^*$ is said to be a `$p$-coresolution of $F$' if the following holds. For all $1\le s <p$, the map $$d^s: C^0\to C^s$$ has kernel $F$, and for all $1\le s <p$ and all $i>0$ the complex
$$C^{i-p+s}\xrightarrow[]{d^{p-s}} C^i \xrightarrow[]{d^{s}} C^{i+s}$$
is exact in $C^i$ (in the complex above, take $C^{j}=0$ if $j<0$).
In particular, if $C^*$ is a $p$-coresolution of $F$, then for all $0\le s<p$ the contraction
$(C_{[s]})^*$ is a coresolution of $F$. Finally, a $p$-complex is acyclic if it is a $p$-coresolution of $0$.

\subsection{Troesch's result}
Recall from section \ref{sec-background} that $\sha_r$ is the graded vector space which equals $\k$ in degrees $i$, for $0\le i<p^r$, and that $V\mapsto S^d(\sha_r\otimes V)$ is considered as a graded strict polynomial functor, as described in example \ref{ex-sha}.
The following theorem is due to Troesch. It generalizes to all prime characteristic some coresolutions which were previously known in characteristic $p=2$ only \cite{FLS,FS}. In the statement of the theorem, we do not give an explicit expression for the differential $\delta$. This expression is actually quite complicated but our arguments do not use the explicit description of $\delta$. For our proofs, we only need to know that $\delta$ exists\footnote{The only exception to this is lemma \ref{lm-prep2} (for which we mention two other proofs which do not rely on Troesch complexes). Everywhere else in the paper, we only use the information on $\delta$ given by theorem \ref{thm-tr}. For the interested reader, we recall Troesch's construction of $B_d(r)^*$ in appendix \ref{sec-app}.}.
\begin{theorem}[{\cite[Thm 2, Thm 4.3.2]{Troesch}}]\label{thm-tr}
Let $r$ be a positive integer and let $n$ be a nonnegative integer. Let $B_d(r)^*$ denote the graded functor $S^d(\sha_r\otimes I)$.
Then $B_d(r)^*$ is equipped with a $p$-differential $\delta$ such that 
\begin{itemize}
\item[(i)] $\delta$ raises the cohomological degree by $p^{r-1}$.
\item[(ii)] If $d=p^rn$ then $(B_d(r)^*, \delta)$ is a $p$-coresolution of $S^{n\,(r)}$. If $p^r$ does not divide $d$, then $(B_d(r)^*, \delta)$ is $p$-acyclic.  
\end{itemize}
\end{theorem}

\begin{remark}\label{rk-summand}
Since the $p$-differential $\delta$ raises the cohomological degree by $p^{r-1}$, $B_d(r)^*$ is the direct sum of the $p$-complexes $B_{d,i}(r)^*$ ($0\le i< p^{r-1}$): 
$$B_{d,i}(r)^*:=\quad  B_d(r)^i\xrightarrow[]{\delta}B_d(r)^{i+p^{r-1}}\xrightarrow[]{\delta}B_d(r)^{i+2p^{r-1}}\xrightarrow[]{\delta}\cdots\;. $$
The meaning of $B_{p^rn}(r)^*$ being a $p$-coresolution of $S^{n\,(r)}$ is that 
$B_{p^rn,0}(r)^*$ is a $p$-coresolution of $S^{n\,(r)}$, and the other $B_{p^rn,i}(r)^*$ are acyclic.
\end{remark}

\subsection{Ready for use injective coresolutions}
For our computations we need to elaborate a bit on theorem \ref{thm-tr}.
For all positive integer $n$ and all $\k$-vector spaces $(V_i)_{1\le i\le n}$, the exponential formula yields an isomorphism of graded objects:
$$B_d(r)^*(V_1\oplus\cdots\oplus V_n)\;\simeq \; \bigoplus_{\mu_i\ge 0\,,\; \sum \mu_i=d} B_{\mu_1}(r)^*(V_1)\otimes \cdots\otimes B_{\mu_n}(r)^*(V_n)\;\;\; (*)$$
By theorem \ref{thm-tr}, the graded object $B_d(r)^*(\bigoplus_{i=1}^n V_i)$ is equipped with a $p$-differential $\delta$. We want to transport $\delta$ on the right hand side of isomorphism $(*)$.
For all $n$-tuple $\mu$ of nonnegative integers of weight $d$, we may consider $\bigotimes_{i=1}^n B_{\mu_i}(r)^*(V_i)$ as a direct summand of $B_d(r)^*(\bigoplus_{i=1}^n V_i)$. 
\begin{lemma}\label{lm-restrict}
The $p$-differential $\delta$ restricts to $\bigotimes_{i=1}^n B_{\mu_i}(r)^*(V_i)$.
\end{lemma}
\begin{proof}
Consider $\bigoplus_{i=1}^n V_i$ as a representation of the $n$-dimensional torus $\mathbb{G}_m^{\times n}$ of weight $(1,\dots,1)$. By functoriality, $B_d(r)^*(\bigoplus_{i=1}^n V_i)$ is a representation of $\mathbb{G}_m^{\times n}$ and $\bigotimes_{i=1}^n B_{\mu_i}^*(r)(V_i)$ is the vector space associated to the weight $(\mu_1,\dots,\mu_n)$. Now $\delta$ is $\mathbb{G}_m^{\times n}$ equivariant, hence $\bigotimes_{i=1}^n B_{\mu_i}(r)^*(V_i)$ is $\delta$-stable.  
\end{proof}

By lemma \ref{lm-restrict}, we can view the $p$-complex 
$(B_d(r)^*(V^{\oplus n}), \delta)$ as the direct sum of the $p$-complexes $(\bigotimes_{i=1}^n B_{\mu_i}(r)^*(V), \delta)$. These latter $p$-complexes serve as a basis for the computations in this article. The remainder of the section is devoted to the study of their properties.

\begin{proposition}[Blackbox 1]\label{prop-res1}Let $r$ be a positive integer and let $\mu=(\mu_1,\dots,\mu_n)$ be a $n$-tuple of nonnegative integers of weight $d$.
Denote by $B_\mu(r)^*$ the graded strict polynomial functor
$$B_\mu(r)^*:= B_{\mu_1}(r)^*\otimes\cdots\otimes B_{\mu_n}(r)^*=S^\mu(\sha_r\otimes I)\;.$$
Then $B_\mu(r)^*$ is equipped with a $p$-differential $\delta$ such that:
\begin{itemize}
\item[(i)] $\delta$ raises the cohomological degree by $p^{r-1}$.
\item[(ii)] If $\mu=p^r\nu$, then $(B_\mu(r)^*,\delta)$ is a $p$-coresolution of $S^{\nu\,(r)}$. If $p^r$ does not divide $\mu$ then, $(B_\mu(r)^*,\delta)$ is $p$-acyclic.
\end{itemize}
\end{proposition}
\begin{proof}
Take for $\delta$ the restriction of the $p$-differential $\delta$ of lemma \ref{lm-restrict} to $B_\mu(r)^*$. Then (i) is satisfied. Moreover $B_\mu(r)^*$ equals $S^\mu$ in degree $0$, so if $p^r$ divides $\mu$ we have an injective map: $S^{\mu/p^r\,(r)}\hookrightarrow B_\mu(r)^0$. Now for all finite dimensional $\k$-vector space $V$,  Troesch complex $B_d(r)^*(V^{\oplus n})$ is isomorphic to the direct sum of complexes $\bigoplus_\mu B_\mu(r)^*(V)$. The inclusion $S^{d/p^r\,(r)}( V^{\oplus n})\hookrightarrow B_d(r)^0(V^{\oplus n})=S^{d}(V^{\oplus n})$ identifies through this isomorphism with the direct sum of the maps $S^{\mu/p^r\,(r)}(V)\hookrightarrow B_\mu(r)^0(V)=S^\mu(V)$ (with the convention that $S^{\mu/p^r\,(r)}(V)=0$ if $p^r$ does not divide $\mu$). Thus (ii) results from theorem \ref{thm-tr}.
\end{proof}

\begin{example}
To make things more concrete, we draw two small examples of $B_\mu(r)^*$. First, if $p=2$, $r=1$, and $\mu=(3)$ then $B_{(3)}(1)^*$ is the $2$-complex:
$$S^6\xrightarrow[]{\delta} S^5\otimes S^1\xrightarrow[]{\delta} S^4\otimes S^2\xrightarrow[]{\delta} S^3\otimes S^3\xrightarrow[]{\delta} S^2\otimes S^4\xrightarrow[]{\delta} S^1\otimes S^5\xrightarrow[]{\delta} S^6\;.$$
Now if $p=3$, $r=1$ and $\mu=(1)$, then $B_{(1)}(1)^*$ is the $3$-complex:
$$S^3\xrightarrow[]{\delta} S^2\otimes S^1 \xrightarrow[]{\delta} \begin{array}{c}S^2\otimes S^1\\\oplus\\ S^1\otimes S^2\end{array}
\xrightarrow[]{\delta} \begin{array}{c}S^3\\\oplus\\ \otimes^3\end{array}\xrightarrow[]{\delta}\begin{array}{c}S^2\otimes S^1\\\oplus\\ S^1\otimes S^2\end{array}
\xrightarrow[]{\delta} S^1\otimes S^2\xrightarrow[]{\delta} S^3$$
In these examples, we do not give an explicit formula for the differential $\delta$, but as we already said  at the beginning of the section, we don't need such a formula for our proofs.
\end{example}

\begin{remark} To define a $p$-differential on $B_\mu(r)^*$, we have chosen to transport the differential $\delta$ of $B_d(r)^*(V^{\oplus d})$ onto the graded object $B_\mu(r)^*(V)$ via isomorphism $(*)$. We could have made another choice, namely we could have taken $B_\mu(r)^*$ as the tensor product of the $p$-complexes $B_{\mu_i}(r)^*$. A priori, there is no reason why these two definitions should coincide. And indeed, one can check from the explicit expression of $\delta$ given in \cite[Section 4.2]{Troesch} that for $r\ge 2$, \emph{these two constructions are not equivalent}, that is, these two constructions induce distinct $p$-differentials on $B_\mu(r)^*$ (cf. remark \ref{rk-iso} for some problems linked with this fact). 
\end{remark}

Let $\mu,\nu$ be two tuples of weight $d$ and let $f\in\hom(S^\mu, S^\nu)$. By evaluating $f$ on $\sha_r\otimes I$, one obtains a graded natural transformation between the graded functors $B_\mu(r)^*=S^\mu(\sha_r\otimes I)$ and $B_\nu(r)^*=S^\nu(\sha_r\otimes I)$. The following statement is specific to our definition of $\delta$.

\begin{proposition}[Blackbox 2]\label{prop-res-2}
Let $r$ be a positive integer, let $\mu,\nu$ be two tuples of nonnegative integers of weight $d$ and let $g\in\hom(S^\mu, S^\nu)$. The morphism of graded functors 
$$g(\sha_r\otimes I) : B_\mu(r)^*\to B_\nu(r)^*$$ 
commutes with the $p$-differential $\delta$ of proposition \ref{prop-res1}. In particular, if $f\in\hom(S^\mu, S^\nu)$ fits into a commutative diagram:
$$\xymatrix{
S^{p^r\mu}\ar[rr]^-{\widetilde{f}}&&S^{p^r\nu}\\
S^{\mu\,(r)}\ar@{^{(}->}[u]\ar[rr]^-{f^{(r)}}&&S^{\nu\,(r)}\ar@{^{(}->}[u]
}$$
then $\widetilde{f}(\sha_r\otimes I)$ is a chain map between the $p$-coresolutions $(B_{p^r\mu}(r)^*,\delta)$ and $(B_{p^r\nu}(r)^*,\delta)$ which lifts $f$. 
\end{proposition}

\begin{proof}{\bf Step 1.}
We first treat the case $S^\mu=S^\nu=\otimes^d$. By lemma \ref{lm-descrHomTT}, the permutations $\sigma\in\Si_d$ form a basis of $\hom(\otimes^d,\otimes^d)$. So we only have to prove the statement for such maps. Let $\sigma\in\Si_d$, then $\sigma$ acts on $(\sha_r\otimes V)^{\otimes d}$ by permuting the factors (without sign) and on $V^{\oplus d}$ by permuting the terms, and for all $V$ we have a commutative diagram of graded vector spaces:
$$\xymatrix{
B_d(r)^*(V^{\oplus d})\ar[d]^-{B_d(r)^*(\sigma)}\ar@{->>}[rr]&& B_{(1,\cdots,1)}(r)^*(V)=(\sha_r\otimes V)^{\otimes d}\ar[d]^-{\sigma(\sha_r\otimes V)}\\ 
B_d(r)^*(V^{\oplus d})\ar@{->>}[rr]&& B_{(1,\cdots,1)}(r)^*(V)=(\sha_r\otimes V)^{\otimes d}.
}$$ 
In this diagram, all the graded objects bear a $p$-differential $\delta$. The horizontal maps are the projections onto a summand of the $p$-complex $B_d(r)^*(V^{\oplus d})$, hence commute with $\delta$. The map $B_d(r)^*(\sigma)$ also commutes with $\delta$, by functoriality of $B_d(r)^*$. Thus, $\sigma(\sha_r\otimes V)$ commutes with $\delta$.

{\bf Step 2.} Next, we treat the case where $S^\mu=\otimes^d$, $\nu$ is an arbitrary $n$-tuple of weight $d$, and where $g:\otimes^d\twoheadrightarrow S^\nu$ is a tensor product of multiplications. Observe that not all the homomorphisms between $\otimes^d$ and $S^\nu$ are of this form (the case of a general morphism $g$ will be treated in the last step of the proof). Let $\Sigma_g:V^{\oplus d}\to V^{\oplus n}$ be the sum map associated to the partition $\nu$, that is $\Sigma_g$ sends $(v_1,\dots,v_d)$ to $(\sum_{i=1}^{\nu_1} v_i,\dots, \sum_{i=d-\nu_n}^{d} v_i)$. Then for all $V$ we have a commutative diagram of graded objects:
$$\xymatrix{
B_d(r)^*(V^{\oplus d})\ar[d]^-{B_d(r)^*(\Sigma_g)}\ar@{->>}[rr]&& B_{(1,\cdots,1)}(r)^*(V)=(\sha_r\otimes V)^{\otimes d}\ar[d]^-{g(\sha_r\otimes V)}\\ 
B_d(r)^*(V^{\oplus n})\ar@{->>}[rr]&& B_{\nu}(r)^*(V)=S^\nu(\sha_r\otimes V).
}$$ 
As before all the the graded objects involved bear a $p$-differential $\delta$, and the horizontal maps and the vertical map on the left commute with $\delta$, whence the result for $g(\sha_r\otimes V)$.

{\bf Step 3.} Finally if $\mu,\nu$ are two tuples of weight $d$, and $g:S^\mu\to S^\nu$, then by projectivity of $\otimes^d$ in $\P_\k$, $g$ lifts to a map $\widetilde{g}:\otimes^d\to\otimes^d$. So for all $V$ we have a commutative diagram of graded objects:
$$\xymatrix{
(\sha_r\otimes V)^{\otimes d}\ar[d]^-{\widetilde{g}(\sha_r\otimes V)}\ar@{->>}[rr]&& S^\mu(\sha_r\otimes V)\ar[d]^-{g(\sha_r\otimes V)}\\ 
(\sha_r\otimes V)^{\otimes d}\ar@{->>}[rr]&& S^\nu(\sha_r\otimes V).
}$$
In this diagram, all the objects bear a $p$-differential. By step 1, $\widetilde{g}(\sha_r\otimes V)$ commutes with the $p$-differentials, and so do the horizontal arrows by step 2. Since the horizontal arrows are epimorphisms, we obtain that $g(\sha_r\otimes V)$ is a morphism of $p$-complexes.
\end{proof}

\section{Classical computations revisited}\label{sec-classical}

Let $r$ be a positive integer, let $\mu$ be a tuple of nonnegative integers and let $F\in\P_\k$. In this section, we focus on the computation of the extension groups of the form:
\begin{align*}(1)\;\Ext^*(F^{(r)}, S^{\mu\,(r)})&& (2)\;  \Ext^*(\Gamma^{\mu\,(r)}, F^{(r)})\;.
\end{align*}
The duality functor $-^\sharp:\P_\k^{\op}\to \P_\k$ of \cite[Prop 2.6]{FS} induces for all $F,G$ an isomorphism, natural in $F,G$, between $\Ext^*(F,G)$ and $\Ext^*(G^\sharp, F^\sharp)$. Therefore we may restrict our attention to $(1)$. 

These extension groups were computed by Cha{\l}upnik in \cite[Thm 4.3]{Chalupnik1}, where the proof relies heavily on the technical computations of \cite{FFSS}. In section \ref{subsec-compute}, we give a new simple direct proof of \cite[Thm 4.3]{Chalupnik1}, relying on Troesch coresolutions. Then we show in section \ref{subsec-more} how to use this result to recover classical computations to be found in the literature.

\subsection{Computation of  $\Ext^*(F^{(r)}, S^{\mu\,(r)})$}\label{subsec-compute}
Our proof relies on an elementary observation: if $T(S^\mu,r)^*$ denotes the injective coresolution of $S^{\mu\,(r)}$ obtained by contracting Troesch $p$-complex $B_{p^r\mu}(r)^*$, then the cochain complex $\hom(F^{(r)},T(S^\mu,r)^*)$ is zero in odd degrees, so its homology is very easy to compute! 
To prove the vanishing of $\hom(F^{(r)},T(S^\mu,r)^*)$ in odd degree, and to compute its even degree part, we first analyze the graded vector space $\hom(F^{(r)}, B_{p^r\mu}(r)^*)$.

\medskip

\noindent
{\bf Step 1: Analysis of the graded vector space $\hom(F^{(r)},B_{p^r\mu}(r)^*)$.} Recall that $B_{p^r\mu}(r)^*$ equals $S^{p^r\mu}(\sha_r\otimes I)$ as a graded functor. By lemma \ref{lm-twist}, evaluation on the Frobenius twist $I^{(r)}$ yields an isomorphism of graded vector spaces:
$$\hom(F,S^\mu(\sha_r^{(r)}\otimes I))\xrightarrow[]{\simeq} \hom(F^{(r)},S^{\mu\,(r)}(\sha_r\otimes I)) \quad(A)$$
Now by evaluating the injection $S^{\mu\,(r)}\hookrightarrow S^{p^r\mu}$ on $\sha_r\otimes I$ we get a graded injection $S^{\mu\,(r)}(\sha_r\otimes I)\hookrightarrow S^{p^r\mu}(\sha_r\otimes I)$, whence a morphism of graded vector spaces:
$$\hom(F^{(r)},S^{\mu\,(r)}(\sha_r\otimes I))\to \hom(F^{(r)},S^{p^r\mu}(\sha_r\otimes I)) \quad(B)$$
which is also an isomorphism by lemma \ref{lm-twist}. We denote by $\xi(F,S^\mu,r)$ the isomorphism of graded vector spaces obtained by composing $(A)$ and $(B)$:
$$\xi(F,S^\mu,r):\; \hom(F,S^\mu(\sha_r^{(r)}\otimes I))\xrightarrow[]{\simeq}\hom(F^{(r)},B_{p^r\mu}(r)^*)\;.$$
The following lemma gives some properties of $\xi(F,S^\mu,r)$.
\begin{lemma}\label{lm-prop-xi}
Let $r$ be a positive integer, let $\mu$ be a tuple of nonnegative integers and let $F\in\P_\k$.
\begin{itemize}
\item[(i)] The graded vector space $\hom(F^{(r)},B_{p^r\mu}(r)^*)$ is concentrated in degrees divisible by $p^r$.
\item[(ii)] Let $f\in\hom(F_1,F_2)$. The following diagram of graded vector spaces is commutative:
$$\xymatrix{
\hom(F_1,S^\mu(\sha_r^{(r)}\otimes I))\ar[rr]^-{\xi(F_1,S^\mu,r)} &&\hom(F_1^{(r)},B_{p^r\mu}(r)^*) \\
\hom(F_2,S^\mu(\sha_r^{(r)}\otimes I))
\ar[u]^-{f^*}\ar[rr]^-{\xi(F_2,S^\mu,r)} &&\hom(F_2^{(r)},B_{p^r\mu}(r)^*)\ar[u]^-{f^{(r)\,*}}
}$$
\item[(iii)] Let $f\in\hom(S^\mu,S^\nu)$, and let $\widetilde{f}:S^{p^r\mu}\to S^{p^r\nu}$ be a lifting of $f^{(r)}$. The following diagram of graded vector spaces is commutative: 
$$\xymatrix{
\hom(F,S^\mu(\sha_r^{(r)}\otimes I))\ar[rr]^-{\xi(F,S^\mu,r)}\ar[d]^{f(\sha_r^{(r)}\otimes I)_*}
 &&\hom(F^{(r)},B_{p^r\mu}(r)^*)\ar[d]^{\widetilde{f}(\sha_r\otimes I)_*}\\
\hom(F,S^\nu(\sha_r^{(r)}\otimes I))\ar[rr]^-{\xi(F,S^\nu,r)} &&\hom(F^{(r)},B_{p^r\nu}(r)^*)
}$$
\end{itemize}
\end{lemma}
\begin{proof}
The graded vector space $\sha_r^{(r)}$, hence the functor $S^\mu(\sha_r^{(r)}\otimes I)$, is concentrated in degrees divisible by $p^r$ (recall that $-^{(r)}$ acts as a homothety of coefficient $p^r$ on the degrees of graded vector spaces). Since $\xi(F,S^\mu,r)$ is a graded isomorphism, we obtain (i). (ii) and (iii) follow directly from the description of $\xi(F,S^\mu,r)$ as the composite of the isomorphisms (A) and (B).
\end{proof}

\noindent
{\bf Step 2: Ext-computation.} 
We first define explicit injective coresolutions $T(S^\mu,r)^*$ of the $S^{\mu\,(r)}$'s by contracting Troesch $p$-complexes.

\begin{definition}\label{def-TSmur}
Recall  Troesch $p$-complex $(B_{p^r\mu}(r)^*,\delta)$ given in proposition \ref{prop-res1}.
We define $T(S^\mu,r)^*$ as the cochain complex such that for all $i\ge 0$:
\begin{align*}
&T(S^\mu,r)^{2i}= B_{p^r\mu}(r)^{p^ri}\;,\\
&T(S^\mu,r)^{2i+1} = B_{p^r\mu}(r)^{p^ri+p^{r-1}}\;,
\end{align*}
the differential $T(S^\mu,r)^{2i}\to T(S^\mu,r)^{2i+1}$ equals $\delta$ and the differential $T(S^\mu,r)^{2i+1}\to T(S^\mu,r)^{2i+2}$ equals $\delta^{p-1}$. By proposition \ref{prop-res1}, $T(S^\mu,r)^*$ is an injective coresolution of $S^{\mu\,(r)}$.
\end{definition}

\begin{example}
In characteristic $p=2$, $T(S^3,1)^*$ has the form:
$$S^6\xrightarrow[]{\delta} S^5\otimes S^1\xrightarrow[]{\delta} S^4\otimes S^2\xrightarrow[]{\delta} S^3\otimes S^3\xrightarrow[]{\delta} S^2\otimes S^4\xrightarrow[]{\delta} S^1\otimes S^5\xrightarrow[]{\delta} S^6\;.$$
In characteristic $p=3$, $T(S^1,1)^*$ has the form:
$$S^3\xrightarrow[]{\delta} S^2\otimes S^1 \xrightarrow[]{\delta^2} S^3\oplus \otimes^3\xrightarrow[]{\delta}
S^2\otimes S^1\oplus S^1\otimes S^2\xrightarrow[]{\delta^2} S^3\;.$$
\end{example}

By definition, $\Ext^*(F^{(r)}, S^{\mu\,(r)})$ is computed as the homology of the cochain complex
$\hom(F^{(r)}, T(S^\mu,r)^*)$. But we have seen in lemma \ref{lm-prop-xi}(i) that the graded vector space $\hom(F^{(r)},B_{p^r\mu}(r)^*)$ is concentrated in degrees divisible by $p^r$. So this means that the complex $\hom(F^{(r)}, T(S^\mu,r)^*)$ is zero in odd degrees! We record this fact in the following lemma.
\begin{lemma}\label{lm-CFSr}
The complex $\hom(F^{(r)},T(S^\mu,r)^*)$ is concentrated in even degrees. So for all $i\ge 0$ we have:
\begin{align*}
&\Ext^{2i+1}(F^{(r)}, S^{\mu\,(r)})=0\;,\\
&\Ext^{2i}(F^{(r)}, S^{\mu\,(r)})= \hom(F^{(r)},T(S^\mu,r)^{2i})= \hom(F^{(r)},B_{p^r\mu}(r)^{p^ri})\;.
\end{align*}
\end{lemma}

Now we can use the isomorphism $\xi(F,S^\mu,r)$ defined in the first step, to identify explicitly the even part of $\Ext^*(F^{(r)}, S^{\mu\,(r)})$. To be more specific, recall that $E_r$ denotes the graded $\k$-vector space concentrated in degrees $2i$ for $0\le i<p^r$, and one dimensional in these degrees. So $S^\mu(E_r\otimes I)$ and $S^\mu(\sha_r^{(r)}\otimes I)$ coincide as ungraded functors, and the degree $2i$ part of the graded functor $S^\mu(E_r\otimes I)$ equals the degree $p^ri$ part of $S^\mu(\sha_r^{(r)}\otimes I)$. Thus, we may rescale $\xi(F,S^\mu,r)$ (i.e. multiply all degrees by $2/p^r$) to get an isomorphism of graded vector spaces $$\hom(F,S^\mu(E_r\otimes I))\simeq \hom(F^{(r)},T(S^\mu,r)^{\mathrm{even}})= \Ext^*(F^{(r)},S^{\mu\,(r)})\;.$$ 
To finish our $\Ext$-computation, we need to check that the isomorphism we have just constructed is natural with respect to $F$ and $S^\mu$.

\begin{lemma}\label{lm-nat}
The isomorphism $\hom(F,S^\mu(E_r\otimes I))\simeq \Ext^*(F^{(r)},S^{\mu\,(r)})$ is natural with respect to $F,S^\mu$.
\end{lemma}
\begin{proof}
Let $f:F\to F'$. Rescaling lemma \ref{lm-prop-xi}(ii) we get a commutative diagram of graded vector spaces (which are zero in odd degree):
$$\xymatrix{
\hom(F,S^\mu(E_r\otimes I))\ar[r]^-{\simeq}& \hom(F^{(r)},T(S^\mu,r)^*)\ar[r]^-{=}&\Ext^*(F^{(r)},S^{\mu\,(r)})\\
\hom(F',S^\mu(E_r\otimes I))\ar[u]^{f^*}\ar[r]^-{\simeq}& \hom({F'}^{(r)},T(S^\mu,r)^*)\ar[u]^{f^{(r)\,*}}\ar[r]^-{=}&\Ext^*({F'}^{(r)},S^{\mu\,(r)})\ar[u]^-{f^{(r)\,*}}
}.$$
This proves the naturality with respect to $F$. The naturality with respect to $S^\mu$ is slightly more delicate. Let $f:S^\mu\to S^\nu$ and let $\widetilde{f}:S^{p^r\mu}\to S^{p^r\nu}$ be a lifting of $f^{(r)}$ (such a lifting exists since $S^{p^r\nu}$ is injective). By proposition \ref{prop-res-2}, $\widetilde{f}(\sha_r\otimes I)$ is a lifting of $f$ to the $p$-coresolutions $B_\mu(r)^*\to B_\nu(r)^*$, so it induces a lifting $T(f,r)^*:=\widetilde{f}(\sha_r\otimes I)_{[1]}$ of $f$ to the injective coresolutions $T(S^\mu,r)^*\to T(S^\nu,r)^*$. Rescaling lemma \ref{lm-prop-xi}(iii) we get a commutative diagram of graded vector spaces (which are zero in odd degree):
$$\xymatrix{
\hom(F,S^\mu(E_r\otimes I))\ar[r]^-{\simeq}\ar[d]^{f(E_r\otimes I)_*}& \hom(F^{(r)},T(S^\mu,r)^*)\ar[d]^{(T(f,r)^*)_*}\ar[r]^-{=}&\Ext^*(F^{(r)},S^{\mu\,(r)})\ar[d]^-{(f^{(r)})_*}\\
\hom(F,S^\nu(E_r\otimes I))\ar[r]^-{\simeq}& \hom({F}^{(r)},T(S^\nu,r)^*)\ar[r]^-{=}&\Ext^*(F^{(r)},S^{\nu\,(r)})
}.$$
This proves the naturality with respect to $S^\mu$.
\end{proof}

To sum up, we have proved (compare \cite[Thm 4.3]{Chalupnik1}):

\begin{theorem}\label{thm-princ}
Let $r$ be a positive integer, let $\mu$ be a tuple of nonnegative integers and let $F\in\P_\k$.
There is a graded isomorphism, natural in $F$ and $S^{\mu}$:
$$\Ext^*(F^{(r)},S^{\mu\,(r)})\simeq \hom(F,S^\mu(E_r\otimes I))\;.$$
Similarly, there is a graded isomorphism, natural in $F$ and $\Gamma^{\mu}$:
$$\Ext^*(\Gamma^{\mu\,(r)},F^{(r)})\simeq \hom(\Gamma^\mu(E_r\otimes I),F)\;.$$
\end{theorem}

\begin{corollary}\label{cor-exact} 
Let $\mu$ be a tuple of nonnegative integers and let $i$ be a nonnegative integer. 
The two following functors are exact: $$F\mapsto \Ext^i(\Gamma^{\mu (r)},F^{(r)})\,,\; \text{ and }\; F\mapsto \Ext^i(F^{(r)}, S^{\mu(r)})\;.$$
\end{corollary}
\begin{proof}
Use that $\Gamma^\mu(E_r\otimes I)$ (resp. $S^\mu(E_r\otimes I)$) is projective (resp. injective).
\end{proof}

\subsection{Other classical $\Ext$-computations}\label{subsec-more}

Now we show that theorem \ref{thm-princ} encompasses concrete $\Ext$-computation which can be found in the literature. First, as a special case we obtain the following corollary, originally proved by Friedlander and Suslin \cite{FS} and which inspires the notation $E_r$. 

\begin{corollary}[{\cite{FS}}]\label{cor-Ext-I-I}
There is an isomorphism of graded vector spaces: 
$$ E_r\simeq \Ext^*(I^{(r)}, I^{(r)})\;.$$
\end{corollary}

Recall that $\Gamma^{d,n}$ denotes the functor $V\mapsto \Gamma^d(\hom_\k(\k^n,V))$, and $S^{d,n}$ its dual. More generally, if $U$ is a finite dimensional vector space, we denote by $\Gamma^{d,U}$ the functor $V\mapsto \Gamma^d(\hom_\k(U,V))$ and by $S^{d,U}$ the functor $V\mapsto S^d(U\otimes V)$. The following result was proved by Cha{\l}upnik \cite[Cor 5.1]{Chalupnik1}.
\begin{corollary}\label{cor-formulemagique} Let $F$ be a homogeneous functor of degree $d$. There are graded isomorphisms, natural in $F,U$:
$$\Ext^*(\Gamma^{d,U\,(r)},F^{(r)})\simeq F(U\otimes E_r)\;,\text{ and }\,\Ext^*(F^{(r)},S^{d,U\,(r)})\simeq F^{\sharp}(U\otimes E_r)\;.$$
\end{corollary}
\begin{proof}
By the exponential formula, $\Gamma^{d,U}$ is a direct sum of functors of the form $\Gamma^\mu$.
So the result follows from theorem \ref{thm-princ} and the graded Yoneda lemma \ref{lm-Yoneda-graded}. 
\end{proof}

\begin{corollary}\label{cor-FFSS} For any $j$, $0\le j\le r$ we have the following computations of $\Ext$-groups:
\begin{itemize}
\item[(i)] {\cite[Thm 4.5]{FS}} The graded vector space
$\Ext^*(I^{(r)}, S^{p^{r-j}\,(j)})$ is concentrated in degrees $q$ such that $q =0 \mod 2p^{r-j}$, $q<2p^r$, and is one dimensional in these degrees.
\item[(ii)] {\cite[Thm 5.1]{FFSS}} Set $U_{r,j}=\Ext^*(\Gamma^{p^{r-j}\,(j)},I^{(r)})$. Then we have isomorphisms
\begin{align*}
&\Ext^*(\Gamma^{dp^{r-j}(j)}, S^{d\,(r)})\simeq S^d(U_{r,j})\\
& \Ext^*(\Gamma^{dp^{r-j}(j)}, \Lambda^{d\,(r)})\simeq \Lambda^d(U_{r,j})
\end{align*}
\end{itemize}
\end{corollary}
\begin{proof}
(i) is a straightforward application of either theorem \ref{thm-princ} or corollary \ref{cor-formulemagique}. For (ii), use that $E_{j}^{(r-j)}=\Ext^*(\Gamma^{p^{r-j}\,(j)},I^{(r)})$ by theorem \ref{thm-princ}, and apply corollary \ref{cor-formulemagique}.
\end{proof}

The symmetric group $\Si_d$ acts on $\otimes^d$ by permuting the factors of the tensor product. This action induces an action of $\Si_d$ on the extension groups $\Ext^*(F^{(r)},\otimes^{d\,(r)})$. As another application of theorem \ref{thm-princ}, we compute the graded $\Si_d$-module $\Ext^*(F^{(r)},\otimes^{d\,(r)})$.
\begin{corollary}\label{cor-belle-formule}
Let $F$ be a strict polynomial functor, let $r,d$ be positive integers. Let $\Si_d$ act on $E_r^{\otimes d}$ by permuting the factors of the tensor product, and let $\Si_d$ act diagonally on $\hom(F,\otimes^d)\otimes E_r^{\otimes d}$. There is a graded isomorphism of $\Si_d$-modules, natural in $F$:
$$\Ext^*(F^{(r)},\otimes^{d\,(r)})\simeq \hom(F,\otimes^d)\otimes E_r^{\otimes d}\;.$$
\end{corollary}
\begin{proof}
By theorem \ref{thm-princ}, $\Ext^*(F^{(r)},\otimes^{d\,(r)})$ is isomorphic (as a $\Si_d$-module) to $\hom(F, (E_r\otimes I)^{\otimes d})$. But $(E_r\otimes I)^{\otimes d}\simeq I^{\otimes d}\otimes E_r^{\otimes d}$ (here $\Si_d$ acts on the left term by permuting the factors of the tensor products, and $\Si_d$ acts on the right term  by permuting simultaneously the factors of the tensor products $I^{\otimes d}$ and $E_r^{\otimes d}$. Now $\hom(F,I^{\otimes d}\otimes E_r^{\otimes d})$ is isomorphic to $\hom(F,I^{\otimes d})\otimes E_r^{\otimes d}$, whence the result.
\end{proof}

\begin{remark} Corollary \ref{cor-belle-formule} is a generalization of the computation of the $\Si_d^{\times 2}$-module $\Ext^*(I^{(r)\,\otimes d},I^{(r)\,\otimes d})$  stated in \cite[p. 781]{Chalupnik1} and also proved in \cite[Thm 1.8]{FF}. 
See also proposition \ref{prop-FFC}.
\end{remark}

\section{Additional structures}\label{subsec-mult}

Extension groups between twisted functors are equipped with the following additional structures.

\begin{enumerate}
\item {\bf Products.}
Let $F_1,F_2,G_1,G_2$ be strict polynomial functors. The extension groups are equipped with a product: 
$$\Ext^{k_1}(F_1,G_1)\otimes \Ext^{k_2}(F_2,G_2)\xrightarrow[]{\cup} \Ext^{k_1+k_2}(F_1\otimes F_2,G_1\otimes G_2)\;, $$
defined on the cochain level by taking tensor products of cocycles, as in e.g. \cite[Section 3.2]{Benson1}. 

\item {\bf Twisting maps.} Let $F,G$ be two strict polynomial functors, and let $r$ be a nonnegative integer.
Evaluation on $I^{(1)}$ yields a twisting map (natural in $F,G$):
$$\mathrm{Fr_1}:\Ext^*(F^{(r)},G^{(r)})\to \Ext^*(F^{(r+1)},G^{(r+1)})\;.$$
\end{enumerate}

The main result of this section is theorem \ref{thm-mult}. It improves theorem \ref{thm-princ} by asserting that one can choose the isomorphisms $\Ext^*(F^{(r)},S^{\mu\,(r)})\simeq \hom(F,S^\mu(E_r\otimes I))$ 
(natural in $F,S^\mu$) so that they are compatible with products and twisting maps. 
To prove theorem \ref{thm-mult}, we first study the special case of $\Ext^*(\otimes^{d\,(r)},\otimes^{d\,(r)})$ in section \ref{subsec-TT}. With this computation in hand, theorem \ref{thm-mult} is a formal consequence the exactness result given in corollary \ref{cor-exact}. Finally, in section \ref{subsec-Hopf-alg-structure}, we apply theorem \ref{thm-mult} to Hopf algebra computations which generalize some results of \cite{FFSS}.

\subsection{Computation of $\Ext^*(\otimes^{d\,(r)},\otimes^{d\,(r)})$}\label{subsec-TT}

Corollary \ref{cor-belle-formule} already yields a description of $\Ext^*(\otimes^{d\,(r)},\otimes^{d\,(r)})$ as a $\Si_d^{\times 2}$-module, but we need a more accurate description (with cup products and precomposition by Frobenius twists).
We first give three preparatory lemmas. The first one is an elementary consequence of lemma \ref{lm-descrHomTT}.
\begin{lemma}\label{lm-prep1}
The following map is an isomorphism: 
$$
\begin{array}{ccc}\hom(I,E_r\otimes I)^{\otimes d}\otimes \k\Si_d&\to &\hom(\otimes^d,(E_r\otimes I)^{\otimes d})\\
f_1\otimes\cdots\otimes f_d\otimes \alpha&\mapsto & \alpha\circ (f_1\otimes\cdots\otimes f_n)\;.
\end{array}$$
\end{lemma}

We have already computed the graded vector space $\Ext^*(I^{(r)},I^{(r)})$ in corollary \ref{cor-Ext-I-I}. We add the description of the twisting map $\mathrm{Fr}_1$ in the following lemma. This result is well known (cf. \cite[Part II, prop 10.14]{Jantzen} or \cite[Cor 4.9]{FS}), although the proofs available are not so easy. We give here a new proof based on Troesch complexes.
\begin{lemma}\label{lm-prep2}
The twisting map $\Ext^*(I^{(r)},I^{(r)})\to \Ext^*(I^{(r+1)},I^{(r+1)})$ is injective. 
\end{lemma}

\begin{proof}
If $J^*$ is an injective coresolution of $I^{(r)}$ and $K^*$ is an injective coresolution of $I^{(r+1)}$, the twisting map is described on the cochain level as the composite
$$\hom(I^{(r)}, J^*)\xrightarrow[]{\mathrm{Fr}_1}\hom(I^{(r+1)}, J^{(1)\,*})\xrightarrow[]{(f)_*} \hom(I^{(r+1)}, K^{*})\;, $$
where $f:J^{(1)\,*}\to K^*$ is a chain map which induces the identity map $I^{(r+1)} =I^{(r+1)}$ after taking homology (such a map exists and is unique up to homotopy).

In this proof, we take the injective coresolutions $J^*=T(I,r)^*$ and $K^*=T(I,r+1)^*$ from definition \ref{def-TSmur}, that is $T(I,r)$ is a contraction of Troesch $p$-complex $B_{p^r}(r)$. Now we give an explicit formula for the chain map $f:T(I,r)^{(1)\,*}\to T(I,r+1)^*$. 
We consider the composite
$$S^{p^{r}}(\sha_{r}\otimes I^{(1)})\simeq S^{p^{r}\,(1)}(\sha_{r}\otimes I) \hookrightarrow S^{p^{r+1}}(\sha_{r}\otimes I) \hookrightarrow S^{p^{r+1}}(\sha_{r+1}\otimes I) \;(*)$$
To be more specific, the first map is induced by the isomorphism $\sha_{r}\simeq \sha_{r}^{(1)}$ which maps the summand $\k$ of degree $i$ of $\sha_{r}$ identically onto the summand $\k$ of degree $pi$ of $\sha_{r}^{(1)}$. The second map is induced by the canonical inclusion $S^{p^{r}\,(1)}\hookrightarrow S^{p^{r+1}}$, and the last map is induced by the canonical inclusion of $\sha_{r}$ into $\sha_{r+1}$. This composite is a morphism of functors, which sends an element of degree $i$ to an element of degree $pi$. Now the graded functors $S^{p^{r}}(\sha_{r}\otimes I^{(1)})=B_{p^{r}}(r)^{*\,(1)}$ and $S^{p^{r+1}}(\sha_{r+1}\otimes I)=B_{p^{r+1}}(r+1)^*$ bear a $p$-differential, and it follows from the construction of the differential of $B_{p^{r+1}}(r+1)^*$ in \cite[Section 4.2 and 4.3]{Troesch} that the composite $(*)$ commutes with the $p$-differentials\footnote{This argument is actually the only argument in the paper where one needs more information on $\delta$ than what is contained in theorem \ref{thm-tr}. For the reader's convenience, we have added a description of Troesch's construction in appendix \ref{sec-app}, with a proof that $f$ is a chain map in lemma \ref{lm-fonct-B}.}. We define $f:T(I,r)^{(1)\,*}\to T(I,r+1)^*$ as the contraction of the composite $(*)$.

Now, with our explicit injective coresolutions and our explicit $f$, we can compute the twisting map. Recall that for all $r\ge 0$, the complex $T(I,r)^*$ contains exactly one copy of $S^{p^r}$ as a direct summand in each even degree strictly less than $2p^r$, so we have an inclusion of graded functors $E_r\otimes S^{p^r}\hookrightarrow T(I,r)^*$. But by definition of $f$, the restriction of $f:T(I,r)^{(1)\,*}\to T(I,r+1)^*$ to $E_r\otimes S^{p^r\,(1)}$ equals the map $\phi$ defined as the composite:
$$E_r\otimes S^{p^r\,(1)}\hookrightarrow E_r\otimes S^{p^{r+1}}\hookrightarrow E_{r+1}\otimes S^{p^{r+1}} \;.$$
Hence we get a commutative diagram of graded vector spaces:{\small
$$\xymatrix{
\hom(I^{(r)},T(I,r)^*)\ar[r]^-{\mathrm{Fr}_1}&\hom(I^{(r+1)},T(I,r)^{(1)\,*})\ar[r]^{f_*}&\hom(I^{(r+1)},T(I,r+1)^{*})\\
\hom(I^{(r)},E_r\otimes S^{p^r})\ar@{^{(}->}[u]\ar[r]^-{\mathrm{Fr}_1}&
\hom(I^{(r+1)},E_r\otimes S^{p^r\,(1)})\ar@{^{(}->}[u]\ar[r]^{\phi_*}&
\hom(I^{(r+1)},E_{r+1}\otimes S^{p^{r+1}})\ar@{^{(}->}[u]
}$$}

Observe that the top row of the diagram actually describes the twisting map on $\Ext$-groups (we don't need to take the map induced in homology since the complexes involved are zero in odd degrees). 
Moreover, we know that $\hom(I^{(r)},T(I,r)^*)\simeq E_r$ for all $r\ge 0$, and by lemma \ref{lm-twist} $\hom(I^{(r)},E_r\otimes S^{p^r})$ is also isomorphic to $E_r$. Since the vertical arrows are injective, this shows that they are actually isomorphisms. Thus we can read the injectivity of the twisting map on the bottom row of the diagram. But $\mathrm{Fr}_1$ is an isomorphism by lemma \ref{lm-twist} and $\phi_*$ is an injection since $\phi$ is injective. This concludes the proof.
\end{proof}

The following lemma is an easy consequence of lemma \ref{lm-prep2}.
\begin{lemma}\label{lm-prep3} One can choose a family of isomorphisms $$\theta(I,I,r):\hom(I,I\otimes E_r)\simeq  \Ext^*(I^{(r)},I^{(r)})$$ for $r\ge 0$ such that for all $r$ we have a commutative diagram (where the left vertical arrow is induced by the canonical inclusion of $E_r$ into $E_{r+1}$):
$$\xymatrix{
\hom(I,I\otimes E_r)\ar[r]^-{\simeq}\ar@{^{(}->}[d]& \Ext^*(I^{(r)},I^{(r)})\ar[d]^{\mathrm{Fr}_1}\\
\hom(I,I\otimes E_{r+1})\ar[r]^-{\simeq}& \Ext^*(I^{(r+1)},I^{(r+1)})
}.$$
\end{lemma}
\begin{proof}
We know that $\mathrm{Fr}_1$ is injective. But since the dimensions of $\Ext^i(I^{(r)},I^{(r)})$ and  $\Ext^i(I^{(r+1)},I^{(r+1)})$ coincide in degrees $i<2p^r$, this means that  $\mathrm{Fr}_1:\Ext^i(I^{(r)},I^{(r)})\to \Ext^i(I^{(r+1)},I^{(r+1)})$ is an isomorphism in degrees $i<2p^r$ and is zero in higher degrees.

So we may define $\theta(I,I,r)$ by induction on $r$. For $r=0$, we take $\theta(I,I,0)=\Id$. For $r\ge 1$, let $\pi_r$ denote the projection of $E_{r+1}$ onto its direct summand $E_r$. It is an isomorphism in degrees $i<2p^r$. We define $\theta(I,I,r+1)$ in degrees $i<2p^r$ as the composite $\mathrm{Fr}_1\circ \theta(I,I,r)\circ \pi_r$. In higher degrees, we can choose any isomorphism between $(E_{r+1})^i$ and $\Ext^i(I^{(r+1)},I^{(r+1)})$ to define $\theta(I,I,r+1)$.
\end{proof}

\begin{proposition}[{Compare \cite[p. 781]{Chalupnik1}, \cite[Thm 1.8]{FF}}]\label{prop-FFC}
Let $d,r$ be positive integers. The following map is an isomorphism:
$$
\begin{array}{ccc}\Ext^*(I^{(r)},I^{(r)})^{\otimes d}\otimes \k\Si_d&\to &\Ext^*(I^{(r)\,\otimes d},I^{(r)\,\otimes d})\\
c_1\otimes\cdots\otimes c_d\otimes \alpha&\mapsto & \alpha_*(c_1\cup\cdots\cup c_n)\;.
\end{array}
$$
Moreover, let $\Si_d^{\times 2}$ act on the source by the formula:
$$(\sigma,\tau)\cdot(c_1\otimes\dots \otimes c_d\otimes \alpha) = c_{\sigma^{-1}(1)}\otimes\dots\otimes c_{\sigma^{-1}(d)}\otimes \tau\alpha\sigma^{-1}\;,$$
and let $\Si_d^{\times 2}$ act on the target by the formula: $(\sigma,\tau)\cdot c = \tau_* (\sigma^{-1})^*c $. Then the above map  is  $\Si_d^{\times 2}$-equivariant.
\end{proposition}
\begin{proof}
The map is clearly equivariant, since it is given by cup products (no sign is involved since $\Ext^*(I^{(r)},I^{(r)})$ is concentrated in even degree). We have to prove it is an isomorphism.
 
We use the injective coresolution  $T(I,r)^*$ of $I^{(r)}$ given in definition \ref{def-TSmur}. Then $T(I,r)^{*\,\otimes d}$ is an injective coresolution of $I^{(r)\,\otimes d}$ and the map of proposition \ref{prop-FFC} is given on the cochain level by the chain map:
$$\begin{array}{cccc}\hom(I^{(r)}, T(I,r)^*)^{\otimes d}\otimes\k\Si_d &\to  &\hom(I^{(r)\,\otimes d},T(I,r)^{\otimes d\,*})&\\
f_1\otimes\dots\otimes f_d\otimes\sigma&\mapsto &\epsilon\, \sigma\circ (f_{1}\otimes\dots\otimes f_{d})&,
\end{array}$$
where $\epsilon$ is a Koszul sign which equals one if all the $f_i$ are in even degree.

By proposition \ref{prop-res1}, each $T(S^1,r)^{k}$ is a direct sum of functors $S^\mu$ of polynomial degree $\sum \mu_i=p^r$. Moreover, for $0\le k<2p^r$, either $k$ is odd and the $S^\mu$  are such that $p^r \not| \mu$, or $k$ is even and $T(S^1,r)^k$ contains exactly one $S^{p^r}$ as a summand. Thus we have a commutative diagram of graded objects:
$$\xymatrix{
\hom(I^{(r)}, T(I,r)^*)^{\otimes d}\otimes\k\Si_d\ar[r]& \hom(I^{(r)\,\otimes d},T(I,r)^{\otimes d\,*})\\
\hom(I^{(r)}, E_r\otimes S^{p^r})^{\otimes d}\otimes\k\Si_d\ar[r]\ar@{^{(}->}[u] &  \hom(I^{(r)\,\otimes d}, (E_r\otimes S^{p^r})^{\otimes d})\ar@{^{(}->}[u]\;.
}$$

By the vanishing lemma \ref{lm-vanish}, the vertical arrows are isomorphisms of graded objects. In particular, the complexes at stake are zero in odd degrees, and the map of proposition \ref{prop-FFC} identifies with the bottom map of the diagram. Now by lemma \ref{lm-twist}, the bottom map of the diagram identifies with the map of lemma \ref{lm-prep1}, hence is an isomorphism.
\end{proof}

\begin{corollary}
There exist a family of isomorphisms $$\theta(\otimes^d,\otimes^d,r):\Ext^*(\otimes^{d\,(r)},\otimes^{d\,(r)})\simeq \hom(\otimes^d,(E_r\otimes I)^{\otimes d})$$
for $d\ge 1$, $r\ge 1$, which are natural in $\otimes^d,\otimes^d$, and compatible with products: $$\theta(\otimes^d,\otimes^d,r)(c)\otimes \theta(\otimes^{d'},\otimes^{d'},r)(c')= \theta(\otimes^{d+d'},\otimes^{d+d'},r)(c\cup c').$$
Moreover, the twisting maps fit into commutative diagrams (where the vertical arrow on the right is induced by the canonical inclusion $E_{r}\hookrightarrow E_{r+1}$)
$$\xymatrix{
\Ext^*(\otimes^{d\,(r)},\otimes^{d\,(r)})\ar[d]^-{\mathrm{Fr}_1}\ar[r]^-{\simeq}& \hom(\otimes^d,(E_r\otimes I)^{\otimes d})\ar@{^{(}->}[d]\\
\Ext^*(\otimes^{d\,(r+1)},\otimes^{d\,(r+1)})\ar[r]^-{\simeq}&\hom(\otimes^d,(E_{r+1}\otimes I)^{\otimes d})\;.
}
$$
\end{corollary}
\begin{proof} Lemma \ref{lm-prep2} yields an isomorphism between $\hom(I,I\otimes E_r)^{\otimes d}\otimes \k\Si_d$ and $\Ext^*(I^{(r)},I^{(r)})^{\otimes d}\otimes \k\Si_d$. We define $\theta(\otimes^d,\otimes^d,r)$ by composing this isomorphism by the inverse of the isomorphism of proposition \ref{prop-FFC} and the isomorphism of lemma \ref{lm-prep1}.
\end{proof}

\subsection{The general case} We now state the main result of section \ref{subsec-mult}.

\begin{theorem}\label{thm-mult}
Let $r$ be a positive integer, let $\mu$ be a tuple of nonnegative integers and let $F\in\P_\k$.
There is a graded isomorphism, natural in $F$ and $S^{\mu}$:
$$\theta(F,S^\mu,r)\;:\Ext^*(F^{(r)},S^{\mu\,(r)})\simeq \hom(F,S^\mu(E_r\otimes I))\;.$$
Moreover, $\theta(F,S^\mu,r)$ satisfies the following properties.
\begin{enumerate}
\item[(1)] $\theta(F,S^\mu,r)$ commutes with products:
$$\theta(F_1\otimes F_2,S^{\mu_1}\otimes S^{\mu_2} ,r)(c_1\cup c_2)=\theta(F_1,S^{\mu_1},r)(c_1)\otimes\theta(F_2,S^{\mu_2},r)(c_2)\;.$$
\item[(2)] $\theta(F,S^\mu,r)$ is compatible with twisting maps, that is 
$\theta(F,S^\mu,r)$ and $\theta(F,S^\mu,r+1)$ fit into a commutative diagram:
$$\xymatrix{
\Ext^*(F^{(r)},S^{\mu\,(r)})\ar[d]^{\mathrm{Fr}_1}\ar[rr]^-{\theta(F,S^\mu,r)}_-{\simeq} &&\hom(F,S^\mu(E_r\otimes I))\ar@{^{(}->}[d]\\
\Ext^*(F^{(r+1)},S^{\mu\,(r+1)})\ar[rr]^-{\theta(F,S^\mu,r+1)}_-{\simeq}&&\hom(F,S^\mu(E_{r+1}\otimes I))\;,
}$$
where the vertical arrow on the right is induced by the canonical inclusion $E_{r}\hookrightarrow E_{r+1}$ ($E_r$ is the part of $E_{r+1}$ of degree $<2p^r$).
\end{enumerate}
Similarly, there is a graded isomorphism, natural in $F$ and $\Gamma^{\mu}$:
$$\Ext^*(\Gamma^{\mu\,(r)},F^{(r)})\simeq \hom(\Gamma^\mu(E_r\otimes I),F)\;,$$
which is compatible with products and twisting maps.
\end{theorem}

\begin{proof}
The duality isomorphism $\Ext^*(F,G)\simeq \Ext^*(G^\sharp,F^\sharp)$ commutes with products and with precomposition by Frobenius twists, so it is sufficient to prove the first part of the theorem (the existence of suitable $\theta(F,S^\mu,r)$, satisfying (1) and (2)). 

To be concise, we let $\mathbb{E}(F,G):=\Ext^*(F^{(r)},G^{(r)})$ and $\mathbb{H}(F,G):=\hom(F,G(E_r\otimes I))$ in this proof. By corollary \ref{cor-exact}, the functors
$F\mapsto \mathbb{E}(F,S^\mu)$ and $G\mapsto \mathbb{E}(\Gamma^\lambda,G)$ are exact. By projectivity of $\Gamma^\lambda(E_r\otimes I)$ and by injectivity of $S^\mu(E_r\otimes I)$, the functors $F\mapsto \mathbb{H}(F,S^\mu)$ and $G\mapsto \mathbb{H}(\Gamma^\lambda,G)$ are also exact. The proof is a formal consequence of this and proposition \ref{prop-FFC}.

Indeed, by proposition \ref{prop-FFC} we already have maps $\theta(\otimes^d,\otimes^d,r)$. But $S^\mu$ admits a presentation $T^\mu_1\to T^\mu_0\twoheadrightarrow S^\mu$ with $T_0^\mu=\otimes^d$ and $T_1^\mu=\bigoplus_{\tau\in\Si_\mu} \otimes^d$ and the map $T_1^\mu\to T_0^\mu$ equals $\sum (\Id-\tau)$ (the sum is taken over the elements of the Young subgroup $\Si_\mu$). Moreover, by projectivity of $\otimes^d$, any map $f:S^\mu\to S^{\mu'}$ lifts to a map of presentations. By exactness of $G\mapsto \mathbb{E}(\otimes^d,G)$ and $G\mapsto \mathbb{H}(\otimes^d,G)$, the presentation $T^\mu_1\to T^\mu_0\twoheadrightarrow S^\mu$ induces presentations of $\mathbb{E}(\otimes^d,S^\mu)$ and $\mathbb{H}(\otimes^d,S^\mu)$, which in turn define maps $\theta(\otimes^d,S^\mu,r)$, natural in $\otimes^d,S^\mu$. Dually, the $\Gamma^\lambda$ have (injective) copresentations by tensor powers, and any functor $F$ has a presentation $P_1\to P_0\twoheadrightarrow F$ by direct sums of $\Gamma^\lambda$. So using these presentations and exactness of $F\mapsto \mathbb{E}(F,S^\mu)$ and $F\mapsto \mathbb{H}(F,S^\mu)$ we define similarly maps $\theta(F,S^\mu,r)$, natural in $F,S^\mu$. The maps $\theta(\Gamma^\lambda,S^\mu,r)$ and $\theta(F,S^\mu,r)$ are characterized as the unique maps fitting into the commutative squares:
$$\xymatrix{
\mathbb{E}(\otimes^d,\otimes^d)\ar@{->>}[d]\ar[rr]^-{\theta(\otimes^d,\otimes^d,r)}&&
\mathbb{H}(\otimes^d,\otimes^d)\ar@{->>}[d]\\
\mathbb{E}(\Gamma^\lambda,S^\mu)\ar[rr]^-{\theta(\Gamma^\lambda,S^\mu,r)}&&
\mathbb{H}(\Gamma^\lambda,S^\mu)
}\quad
\xymatrix{
\mathbb{E}(F,S^\mu)\ar@{^{(}->}[d]\ar[rr]^-{\theta(F,S^\mu,r)}&&
\mathbb{H}(F,S^\mu)\ar@{^{(}->}[d]\\
\mathbb{E}(P_0,S^\mu)\ar[rr]^-{\theta(P_0,S^\mu,r)}&&
\mathbb{H}(P_0,S^\mu)\;.
}$$
(One easily checks that $\theta(F,S^\mu,r)$ does not depend on the choice of $P_0$). 

It remains to check compatibility with cup products and twisting maps. But using the naturality of cup products (resp. twisting maps) and the two commutative squares above, one can reduce the compatibility for general $\theta(F,S^\mu,r)$ to the compatibility for $\theta(\Gamma^\lambda,S^\mu,r)$, which in turn reduces to the compatibility for $\theta(\otimes^d,\otimes^d,r)$. The latter holds by proposition \ref{prop-FFC}.
\end{proof}

\begin{remark}\label{rk-iso}
We don't know if the isomorphisms of theorems \ref{thm-princ} and \ref{thm-mult} are equal in general. Looking at the proof of theorem \ref{thm-mult}, one  sees that they coincide if and only if they coincide for the special case $F=S^\mu=\otimes^d$. 

In this special case, the isomorphism of theorem \ref{thm-princ} is built by using the injective coresolution $T(\otimes^d,r)^*$, while the isomorphism of theorem of  \ref{thm-mult} uses $(T(I,r)^*)^{\otimes d}$. These two injective coresolutions are not equal in general (if $p\ne 2$ of if $d\ge 2$, they do not coincide as graded objects). But in both cases, the coresolution contains the graded object $(E_r\otimes S^{p^r})^{\otimes d}$. And the isomorphisms of theorems \ref{thm-princ} and \ref{thm-mult} follow in both cases from the fact that the complex obtained after applying $\hom(\otimes^{d\,(r)},-)$ to the coresolution equals the graded vector space $\hom(\otimes^{d\,(r)}, (E_r\otimes S^{p^r})^{\otimes d})$, with zero differential.
So, to prove that the isomorphisms of theorems \ref{thm-princ} and \ref{thm-mult} are equal, we have to find a homotopy equivalence $h^*:T(\otimes^d,r)^*\to (T(I,r)^*)^{\otimes d}$ whose restriction to $(E_r\otimes S^{p^r})^{\otimes d}$ is the identity. 

If $r=1$, $B_{(1,\dots,1)}(1)^*\simeq (B_{1}(1)^*)^{\otimes d}$ by \cite[Prop 3.2.1]{Troesch}, so \cite[Prop 2.4]{TouzeUniv} yields such a homotopy equivalence. Hence the isomorphisms of theorems \ref{thm-princ} and \ref{thm-mult} coincide for $r=1$. 

If $r>1$, $B_{(1,\dots,1)}(r)^*$ and $(B_{1}(r)^*)^{\otimes d}$ are equal as graded objects, but they don't bear the same differential, so we cannot conclude the equality of the isomorphisms by the previous method, and the question wether the isomorphisms of theorems \ref{thm-princ} and \ref{thm-mult} coincide is open.
\end{remark}

\subsection{Applications to algebra and Hopf algebra structures}\label{subsec-Hopf-alg-structure}

Let $U$ be a finite dimensional $\k$-vector space. 
Recall that $\Gamma^{d,U}$ denotes the functor $V\mapsto \Gamma^{d}(\hom_\k(U,V))$.
Let $F_i$, $i=1,2$, be strict polynomial functors, and let $d_1,d_2$ be nonnegative integers. We define a morphism:
\begin{align*}\textstyle\bigotimes_{i=1}^2\Ext^*(\Gamma^{d_i,U\,(r)},F_i)&\xrightarrow[]{\cup} \Ext^*((\Gamma^{d_1,U}\otimes\Gamma^{d_2,U})^{(r)},(F_1\otimes F_2)^{(r)})\\ &\qquad\quad\xrightarrow[]{(\Delta^{U\,(r)}_{d_1,d_2})^*} \Ext^*(\Gamma^{d_1+d_2,U\,(r)},(F_1\otimes F_2)^{(r)})\;, 
\end{align*}
where $\Delta_{d_1,d_2}^U: \Gamma^{d_1+d_2,U}\to \Gamma^{d_1,U}\otimes \Gamma^{d_2,U}$ is the evaluation of the comultiplication $\Delta_{d_1,d_2}: \Gamma^{d_1+d_2}\to \Gamma^{d_1}\otimes \Gamma^{d_2}$ of the divided power algebra 
on the functor $\hom_\k(U,-)$. 
The following statement provides an extra compatibility property for Cha{\l}upnik's formula \cite[Cor 5.1]{Chalupnik1}. 
\begin{corollary}\label{cor-iso}
Let $F$ be a homogeneous strict polynomial functor of degree $d$. There are isomorphisms, natural in $F,U$:
$$\Ext^*(\Gamma^{d,U\,(r)},F)\simeq F(U\otimes E_r)\;.$$
Moreover, if  $F_i$, $i=1,2$ are homogeneous strict polynomial functors of respective degree $d_i$, these isomorphisms fit into commutative diagrams:
$$\xymatrix{
\bigotimes_{i=1}^2\Ext^*(\Gamma^{d_i,U\,(r)},F_i)\ar[rr]^-{\simeq}\ar[d]^-{c_1\otimes c_2\mapsto (\Delta_{d_1,d_2}^{(r)})^*(c_1\cup c_2)}
&& F_1(U\otimes E_r)\otimes F_2(U\otimes E_r)\ar[d]^-{=}\\
\Ext^*(\Gamma^{d_1+d_2,U\,(r)},(F_1\otimes F_2)^{(r)})\ar[rr]^-{\simeq}&&(F_1\otimes F_2)(U\otimes E_r)\;.
}$$
In particular, the vertical arrow on the left is an isomorphism.
\end{corollary}
\begin{proof}
 By theorem \ref{thm-mult}, $c_1\otimes c_2\mapsto (\Delta_{d_1,d_2}^{U\,(r)})^*(c_1\cup c_2)$ identifies with the map 
 $$
\begin{array}{ccc}\bigotimes_{i=1}^2\hom(\Gamma^{d_i,U}(E_r\otimes I),F_i)&\xrightarrow[]{\simeq} &\hom(\Gamma^{d_1+d_2,U}(E_r\otimes I),F_1\otimes F_2)\quad (*)\\
f_1\otimes f_2&\mapsto & (f_1\otimes f_2)\circ \Delta_{d_1,d_2}^U(E_r\otimes I)\;.
\end{array}
$$
For all $F$, the graded Yoneda isomorphism $\hom(\Gamma^d((V^*)^\vee\otimes I),F)\simeq F(V^*)$ of lemma \ref{lm-Yoneda-graded} is nothing but evaluation on $\Id_{V^*}^{\otimes d}\in (V^*)^\vee\otimes V^*$. So the map $(*)$ identifies through Yoneda isomorphisms with the identity map $F_1(U\otimes E_r)\otimes F_2(U\otimes E_r) = (F_1\otimes F_2)(U\otimes E_r)$. 
\end{proof}

Let $(A^j)_{j\in\mathbb{N}}$ be a family of strict polynomial functors, and assume that $A^*$ is endowed with a graded algebra structure (i.e. we have maps of strict polynomial functors $A^k\otimes A^\ell\to A^{k+\ell}$ and $\k\to A^0$ which satisfy the axioms of an algebra).  In such a situation, the following holds
\begin{enumerate} 
\item Each $A^j$ splits as a finite direct sum of homogeneous functors of degree $d$: $A^j=\bigoplus_{d\ge 0} A_d^j$, so that the graded functor $A^*=\bigoplus_{d,j}A^j_d$ is actually automatically bigraded. Homogeneity in $\P_\k$ \cite[Prop. 2.6]{FS} implies that $A^*$ is a bigraded algebra.
\item The extension groups $\bigoplus_{i,d,j}\Ext^i(\Gamma^{d,U\,(r)}, A^{j\,(r)})$ are equipped with a trigraded algebra structure with unit $\k=\Ext^{0}(\k^{(r)},\k^{(r)})\to \Ext^{0}(\k^{(r)}, A^{0\,(r)})
$ and with multiplication defined as the composite (the first map is the map $c_1\otimes c_2\mapsto (\Delta_{d_1,d_2}^{U\,(r)})^*(c_1\cup c_2)$ from corollary \ref{cor-iso}, the second one is induced by the multiplication of $A^*$)
\begin{align*}\textstyle\bigotimes_{k=1}^2 \Ext^{i_k}(\Gamma^{d_k,U\,(r)}, A^{j_k\,(r)})&\to \Ext^{i_1+i_2}(\Gamma^{d_1+d_2,U\,(r)}, A^{j_1\,(r)}\otimes A^{j_2\,(r)}) \\
&\to \Ext^{i_1+i_2}(\Gamma^{d_1+d_2,U\,(r)}, A^{j_1+j_2\,(r)}) 
\end{align*}
\end{enumerate}
The following result is straightforward from corollary \ref{cor-iso}.

\begin{corollary}\label{cor-alg}
Let $A^*$ be a family of strict polynomial functors endowed with a graded algebra structure. Let $U$ be a finite dimensional vector space. 
There is an isomorphism of trigraded algebras:
$$\bigoplus_{i,d,j}\Ext^i(\Gamma^{d,U\,(r)}, A^{j\,(r)})\simeq \bigoplus_{i,d,j} A_d^j(U\otimes E_r)\;.$$
Here, $A_d^j$ denotes the homogeneous degree $d$ part of $A^j$, and the degree $i$ on the right handside corresponds to the degree which arises when we evaluate the functor $A_d^j$ on the graded vector space $U\otimes E_r$ ($U$ placed in degree $0$).
\end{corollary}

Similarly, if $C^*$ is a family of strict polynomial functors endowed with the structure of a graded coalgebra, then $\bigoplus_{i,d,j}\Ext^i(\Gamma^{d,U\,(r)}, C^{j\,(r)})$ is endowed with the structure of a trigraded coalgebra, whose comultiplication equals the composite (the isomorphism is inverse to the one given by corollary \ref{cor-iso}):
\begin{align*}
\Ext^{i}(\Gamma^{d,U\,(r)}, C^{j\,(r)})&\to \textstyle \bigoplus_{\text{\tiny$
\begin{array}{c}j_1+j_2=j\\d_1+d_2=d\end{array}$}
}\Ext^{i}(\Gamma^{d,U\,(r)}, C_{d_1}^{j_1\,(r)}\otimes C_{d_2}^{j_2\,(r)})\\
&\simeq  \textstyle \bigoplus_{\text{\tiny$
\begin{array}{c}i_1+i_2=i\\j_1+j_2=j\\d_1+d_2=d\end{array}$}
}\textstyle\bigotimes_{k=1}^2 \Ext^{i_k}(\Gamma^{d_k,U\,(r)}, C_{d_k}^{j_k\,(r)})\;.
\end{align*}
And we similarly obtain a trigraded coalgebra isomorphism:
$$\bigoplus_{i,d,j}\Ext^i(\Gamma^{d,U\,(r)}, C^{j\,(r)})\simeq \bigoplus_{i,d,j} C_d^j(U\otimes E_r)\;.$$
Now we want to combine the two previous structures, that is, we consider a family $ H^*$ of strict polynomial functors endowed with a graded Hopf algebra structure without antipode (The case with antipode works similarly and we leave the slight modifications to the interested reader).

There is a trap here: $\bigoplus_{i,d,j}\Ext^i(\Gamma^{d,U\,(r)}, H^{j\,(r)})$ is a trigraded algebra and a trigraded coalgebra, but not a trigraded Hopf algebra in general. There is a sign problem with the partial grading corresponding to the letter `$d$' when one wants to check the compatibility axiom between the multiplication and the comultiplication. This problem can be fixed in two ways. Either double the gradations, that is consider $\Gamma^{d, U\,(r)}$ in degree $2d$, or forget the $d$-grading, that is consider the bigraded algebra structure. We choose this latter solution.   

It is easy to check from our definitions that the coalgebra and the algebra structure of $\bigoplus_{i,d,j}\Ext^i(\Gamma^{d,U\,(r)}, H^{j\,(r)})$ satisfy the axioms of a $(i,j)$-bigraded Hopf algebra (without antipode) (this is also proved when $H$ is a `Hopf exponential functor' in \cite[Lemma 1.11]{FFSS} and in general in \cite[Thm 5.16]{TouzeClassical}). Then, from our description of algebra and coalgebra structures we get the following result (which generalizes \cite[Thm 5.8 (1,2,5)]{FFSS}).

\begin{corollary}
Let $H^*$ be a family of strict polynomial functors endowed with the structure of a graded Hopf algebra. There is an a isomorphism of $(i,j)$-bigraded Hopf algebras:
$$\bigoplus_{i,d,j}\Ext^i(\Gamma^{d,U\,(r)},H^{j\,(r)})\simeq \bigoplus_{i,j} H^j(U\otimes E_r) \;.$$
Here, the degree $i$ on the right handside corresponds to the degree which arises when we evaluate the functor $H^j$ on the graded vector space $U\otimes E_r$ (with $U$ placed in degree $0$).
\end{corollary}

\begin{example}
Consider the graded Hopf algebra $H^*$ defined for all $j\ge 0$ by $H^{2j+1}:=0$ and $H^{2j}:= S^j(S^2)$. Then, as a $(i,j)$-bigraded Hopf algebra,
$\bigoplus_{i,d,j}\Ext^i(\Gamma^{d\,(r)},S^j(S^{2\,(r)}))$ is a symmetric Hopf algebra on generators $\alpha_i,\beta_j$ and $\gamma_{k,\ell}$, where $0\le i< p^r$, $0\le j<p^r$ and $0\le k<\ell<p^r$, placed in respective bidegrees $(4i,2)$, $(4j,2)$ and $(2(k+\ell),2)$.
\end{example}

\section{The semi-stable range}\label{sec-semistable}

If $F$ and $G$ are strict polynomial functors of degree $d$, the evaluation map: 
$$\mathrm{ev}_n:\Ext^*_{\P_\k}(F,G)\to \Ext^*_{GL_n}(F(\k^n),G(\k^n))$$
is an isomorphism if $n\ge d$ \cite[Cor 3.13]{FS}. We call such values of $n$ the `stable range' (relative to $F,G$). In general, if $n$ is not in the stable range (i.e. $n<d$), the evaluation map may very well not be an isomorphism, as the following example shows it.

\begin{example}\label{ex-calc} Let $p=2$, $d=2$, and let $F=I^{(1)}$ and $G=\Gamma^2$. Then 
$$\Ext^i_{\P_\k}(I^{(1)},\Gamma^2)= \k \text{ if $i=2$, and $0$ otherwise.}$$
On the other hand, if we evaluate on $\k$ then $\k^{(1)}=\Gamma^2(\k)$ is the one dimensional $GL_1=\mathbb{G}_m$ representation of weight $2$, so that we have:
$$\Ext^*_{GL_1}(\k^{(1)},\Gamma^2(\k))=\hom_{GL_1}(\k^{(1)},\Gamma^2(\k))=\k\;.$$
So the evaluation map $\mathrm{ev}_n$ is neither injective, nor surjective.
\end{example}
\begin{proof}
To compute $\Ext^*_{\P_\k}(I^{(1)},\Gamma^2)$, we know that $\Ext^*_{\P_\k}(I^{(1)}, I^{(1)})=E_1$ (by corollary \ref{cor-Ext-I-I}) and that $\Ext^*_{\P_\k}(I^{(1)},S^2)$ equals $\k$, concentrated in degree $0$ (by duality \cite[Prop 2.6]{FS} and by the Yoneda lemma \cite[Thm 2.10]{FS}). Then we use the short exact sequences $I^{(1)}\hookrightarrow S^2\twoheadrightarrow \Lambda^2$ and 
$\Lambda^2\hookrightarrow \Gamma^2\twoheadrightarrow I^{(1)}$.
\end{proof}

Unfortunately, the stable range is not always sufficient for applications to the rational cohomology of $GL_n$. For example, we have computed in corollary \ref{cor-alg} the $\Ext$-algebras $\Ext^*_{\P_\k}(\Gamma^{*,U},A^*)$. One can wonder if the evaluation map 
$$\Ext^*_{\P_\k}(\Gamma^{*,U\,(r)},A^{*\,(r)})\xrightarrow[]{\mathrm{ev}_n}\Ext^*_{GL_n}(\Gamma^{*,U}(\k^{n\,(r)}),A^*(\k^{n\,(r)}))\quad (\star)$$ 
is an isomorphism for some values of $n$. This is not the case in general. For example, take the case of $U=\k$ and $A^*=S^*$, then Friedlander and Suslin's result \cite[Cor 3.13]{FS} shows that $\mathrm{ev}_n$ is an isomorphism in tridegrees $(i,j,k)$ with $j,k\le n/p^r$. But for higher tridegrees, we are out of the stable range, so we don't know  that $\mathrm{ev}_n$ is an isomorphism in these degrees.

The purpose of this section is to improve Friedlander and Suslin's range of values of $n$ for which the evaluation map $\mathrm{ev}_n$ is an isomorphism. We call the values of $n$ such that $n < d$ (i.e. out of Friedlander and Suslin's stable range) but for which $\mathrm{ev}_n:\Ext^*_{\P_\k}(F,G)\to \Ext^*_{GL_n}(F(\k^n),G(\k^n))$ is still an isomorphism `the semi-stable range' (relative to $F,G$). 

The section is organized as follows. In section \ref{subsec-review}, we review some material from \cite[Section 3]{FS}. In section \ref{subsec-n-resolved}, we give conditions on $F$ and $G$  so that the semi-stable range relative to $F,G$ is not trivial. With the help of Troesch complexes, we prove that these conditions are satisfied for $F=\Gamma^{n\,(r)}$ (This is proposition \ref{prop-ex-nresolved}(i)). In particular, the evaluation map $(\star)$ above is an isomorphism for $n\ge p^r\dim U$. In section \ref{subsec-FFT-SFT} we use this fact to prove in theorem \ref{cor-CFFT} a cohomological version of the First Fundamental Theorem (FFT) and Second Fundamental Theorem (SFT) for $GL_n$, for $n$ big enough. In question \ref{qn-CFFT}, we give a conjecture for small values of $n$ (see also remark \ref{rk-SFT}).

\subsection{Review of section 3 of {\cite{FS}}}\label{subsec-review}
We keep the notations of \cite{FS}, so  
$S(n,d)=\Gamma^d(\hom_\k(\k^n,\k^n))$ is the Schur algebra and $(\mathrm{Pol})_{n,d}$ is the full subcategory of the category of rational $GL_n$-modules consisting of homogeneous polynomial modules of degree $d$, which are finite dimensional. Using the adjunction isomorphism: 
$$\hom_\k(\Gamma^d(\hom_\k(\k^n,\k^n))\otimes V, V)\simeq  \hom_\k(V, S^d(\hom_\k(\k^n,\k^n)^\vee)\otimes V) $$
one sees that being a left $S(n,d)$-module is equivalent to being a $S^d(\hom_\k(\k^n,\k^n)^\vee)$-comodule, that is, a homogeneous polynomial $GL_n$-module of degree $d$. Whence an equivalence (it is even an isomorphism) of categories (for all $n>0$ and all $d\ge 0$):
$$\mathrm{Mod}\{S(n,d)\}\simeq (\mathrm{Pol})_{n,d}\qquad(1)$$
(By definition, the modules of $\mathrm{Mod}\{S(n,d)\}$ are finite dimensional). By \cite[Cor 3.12.1]{FS} one knows that (again for all $n>0$ and all $d\ge 0$) the embedding of categories
$$(\mathrm{Pol})_{n,d}\hookrightarrow \text{ rational $GL_n$-modules}\qquad(2)$$
induces an isomorphism on $\Ext$-groups.
If $F\in\P_{\k,d}$ is homogeneous of degree $d$, then for all $V,W\in\mathcal{V}_\k$ we have a homogeneous polynomial map of degree $d$: $\hom_\k(V,W)\to \hom_\k(F(V),F(W))$, or equivalently a $\k$-linear map $\Gamma^d(\hom_\k(V,W))\to \hom_\k(F(V),F(W))$. If we take $V=W=\k^n$, this $\k$-linear map provides a $S(n,d)$-module structure on $F(\k^n)$, so we have an evaluation functor (for all $n>0$ and all $d\ge 0$):
$$\P_{\k,d}\to \mathrm{Mod}\{S(n,d)\}\;,\quad F\mapsto F(\k^n)\qquad(3)$$
Finally, the evaluation functor $\P_{\k,d}\to \text{ rational $GL_n$-modules}$, $F\mapsto F(\k^n)$ equals the composite of the functors $(1)$, $(2)$ and $(3)$:
$$\P_{\k,d}\to \mathrm{Mod}\{S(n,d)\}\simeq (\mathrm{Pol})_{n,d}\hookrightarrow \text{ rational $GL_n$-modules}\;.$$
We sum up the situation in the following proposition.
\begin{proposition}
For all $n>0$ and all $d\ge 0$, the map $\mathrm{ev}_n$ induced by evaluation on the standard representation of $GL_n$:
$$\mathrm{ev}_n:\Ext^*_{\P_{\k,d}}(F,G)\to \Ext^*_{GL_n}(F(\k^n),G(\k^n))$$
decomposes as:
$$\Ext^*_{\P_{\k,d}}(F,G)\xrightarrow[]{} \Ext^*_{S(n,d)}(F(\k^n),G(\k^n))\simeq \Ext^*_{GL_n}(F(\k^n),G(\k^n))\;.$$
\end{proposition}
Friedlander and Suslin prove \cite[Thm 3.2]{FS} that if $n\ge d$, the evaluation functor $\P_{\k,d}\xrightarrow[]{(3)} \mathrm{Mod}\{S(n,d)\}$ is an equivalence of categories, but we are interested in the general case, where this is no longer true. We shall prove below that the evaluation functor induces an isomorphism on the level of $\Ext$-groups when $F$ (or $G$) admit special projective resolutions (or injective coresolutions). We call such functors `$n$-resolved' and '$n$-coresolved'.

\subsection{$n$-resolved functors}\label{subsec-n-resolved}
The notion of $n$-resolved functors is a natural generalization of $n$-generated functors, defined in \cite{FS}. So we first recall the definition of $n$-generated functors from \cite{FS}.

\begin{definition}[{\cite[Thm 2.10]{FS}}]
Let $F\in\P_{\k,d}$ and let $n$ be a positive integer. Recall the map $\theta_F:\Gamma^{d,n}\otimes F(\k^n)\to F$ defined 
for all $V\in\mathcal{V}_\k$ by: 
$$\Gamma^d(\hom_\k(\k^n,V))\otimes F(\k^n)\to F(V)\,,\qquad f\otimes x\mapsto F(f)(x)\;.$$
We call $F$ $n$-generated if $\theta_F$ is an epimorphism.
\end{definition}

\begin{example} $\Gamma^{d,n}$ itself is $n$-generated.
\end{example}

\begin{lemma}[Properties of $n$-generation]\label{lm-prop-n-gen}
~

\begin{itemize}
\item[(i)] A functor $F$ is $n$-generated if and only if there is a $n$-generated projective $P$ and an epimorphism $P\twoheadrightarrow F$.

In particular, a direct summand or a quotient of a $n$-generated functor is again $n$-generated.
\item[(ii)] If $F$ is $n$-generated then $F$ is $(n+1)$-generated.
\item[(iii)] If $F$ has degree $d$ then $F$ is $d$-generated.
\item[(iv)] If $F_i$ , $i=1,2$ is $n_i$-generated then $F_1\otimes F_2$ is $(n_1+n_2)$-generated.
\item[(v)] If $F$ is $n$-generated and $\ell$ is a positive integer, then $F(\k^\ell\otimes I)$ is $n\ell$-generated.
\item[(vi)] If $F$ is $n$-generated then $F^{(r)}$ is $n$-generated 
\end{itemize}
\end{lemma}
\begin{proof}
For (i), the `only if' part is the definition, for the `if' part, use the commutative diagram:
$$\xymatrix{
\Gamma^{d,n}\otimes P(\k^n)\ar@{->>}[r]^-{\theta_P}\ar@{->>}[d]& P\ar@{->>}[d]\\
\Gamma^{d,n}\otimes F(\k^n)\ar@{->}[r]^-{\theta_F}& F\;.\\
}$$
For (ii) use that $\Gamma^{d,n}$ is a quotient of the $(n+1)$-generated functor
$\Gamma^{d,n+1}$. (iii) is just \cite[Prop 2.9]{FS}. For (iv), use that $\Gamma^{d_1,n_1}\otimes \Gamma^{d_2,n_2}$ is a quotient of $\Gamma^{d_1+d_2,n_1+n_2}$ by the exponential formula. For (v), use that $\Gamma^{d,n}(\k^\ell\otimes I)\simeq \Gamma^{d,n\ell}$. Finally, (vi) follows from the epimorphism $\Gamma^{dp^r,n}\twoheadrightarrow \Gamma^{d,n\,(r)}$.
\end{proof}

The following lemma is the key result for our applications.

\begin{lemma}\label{lm-projngen}
Let $P$ be a $n$-generated projective of $\P_{\k,d}$. 
\begin{itemize}
\item[(i)] $P(\k^n)$ is projective in the category $\mathrm{Mod}\{S(n,d)\}$
\item[(ii)] For all $G\in\P_{\k,d}$ the evaluation map induces an isomorphism
$$\hom_{\P_\k}(P,G)\simeq \hom_{S(n,d)}(P(\k^n), G(\k^n))\;.$$
\end{itemize}
\end{lemma}
\begin{proof}
Since $n$-generated projectives are direct summands of finite direct sums of $\Gamma^{d,n}$, it suffices to prove the lemma for $P=\Gamma^{d,n}$. Then (i) is trivial since $\Gamma^{d,n}(\k^n)= S(n,d)$, and (ii) follows from the commutative diagram:
$$\xymatrix{
\hom_{\P_\k}(\Gamma^{d,n},G)\ar[r]\ar[d]^-{\simeq} & \hom_{S(n,d)}(\Gamma^{d,n}(\k^n), G(\k^n))\ar[d]^-{\simeq}\\
G(\k^n)\ar[r]^-{=}&G(\k^n)
}$$
where the vertical isomorphism on the left is the Yoneda lemma (which is nothing but evaluation on $\Id_{\k^n}^{\otimes d}\in \Gamma^{d,n}(\k^n)$), and the vertical isomorphism on the right is evaluation on $\Id_{\k^n}^{\otimes d}$ which is the unit of $\Gamma^{d,n}(\k^n)=S(n,d)$. 
\end{proof}

Now we briefly recall the dual notion of $n$-cogenerated functors. Let $S^{d,n}=(\Gamma^{d,n})^\sharp$ be the standard injective. By duality, the map $\theta_{F^\sharp}:\Gamma^{d,n}\otimes F^\sharp(\k^n)\to F^\sharp$ yields a map $\theta_{F^\sharp}^\sharp: F\to S^{d,n}\otimes F((\k^n)^\vee)$.
\begin{definition}
Let $F\in\P_{\k,d}$. We call $F$ $n$-cogenerated if $F^\sharp$ is $n$-generated, or equivalently if the map $\theta_{F^\sharp}^\sharp: F\to S^{d,n}\otimes F((\k^n)^\vee)$ is a monomorphism.
\end{definition}

By duality, one sees that $S^{d,n}$ is $n$-cogenerated. Also, lemma \ref{lm-prop-n-gen} dualizes easily. Now we have an analogue of lemma \ref{lm-projngen}:
\begin{lemma}\label{lm-injncogen}
Let $J\in\P_{\k,d}$ denote a $n$-cogenerated injective.
\begin{itemize}
\item[(i)] $J(\k^n)$ is injective in the category $\mathrm{Mod}\{S(n,d)\}$
\item[(ii)] For all $F\in\P_{\k,d}$ the evaluation map induces an isomorphism
$$\hom_{\P_\k}(F,J)\simeq \hom_{S(n,d)}(F(\k^n), J(\k^n))\;.$$
\end{itemize}
\end{lemma}
\begin{proof} First, it is easy to prove that lemma \ref{lm-projngen} has an analogue in the category $\{S(n,d)\}\mathrm{Mod}$ of finite dimensional \emph{right} $S(n,d)$-modules. Now if $M$ is a left $S(n,d)$-module, it dual $M^\vee$ is naturally a right $S(n,d)$-module, and we have a commutative diagram, natural in $F,G$:
$$\xymatrix{
\hom_{\P_\k}(F,G)\ar[d]\ar[r]^-{\simeq}&\hom_{\P_\k}(G^\sharp,F^\sharp)\ar[d] \\
\hom_{\mathrm{Mod}\{S(n,d)\}}(F(\k^n), G(\k^n))\ar[r]^-{\simeq}&\hom_{\{S(n,d)\}\mathrm{Mod}}(G(\k^n)^\vee, F(\k^n)^\vee)
}$$
where the vertical arrow on the right is evaluation on the right $S(n,d)$-module $\k^{n\,\vee}$. So lemma \ref{lm-injncogen} is simply a translation of the right-$S(n,d)$-module analogue of lemma \ref{lm-projngen}.
\end{proof}

\begin{definition}
Let $F\in\P_{\k,d}$ and let $n$ be a positive integer. We say that $F$ is $n$-resolved if $F$ admits a resolution by $n$-generated projectives. We say that $F$ is n-coresolved if $F$ admits a coresolution by $n$-cogenerated injectives (or equivalently if $F^\sharp$ is $n$-resolved).
\end{definition}

The interest of $n$-(co)resolved functors lies in the following theorem, which is an immediate consequence of lemmas \ref{lm-projngen} and \ref{lm-injncogen}.

\begin{theorem}\label{thm-nres-iso}
Let $F,G\in \P_{\k,d}$, and let $n$ be an integer such that $F$ is $n$-resolved, or $G$ is $n$-coresolved. Then the evaluation map is an isomorphism:
$$\Ext^*_{\P_\k}(F,G)\xrightarrow[]{\simeq} \Ext^*_{GL_n}(F(\k^n),G(\k^n))\;. $$
\end{theorem}
\begin{proof}Since the evaluation map decomposes as
$$\Ext^*_{\P_{\k,d}}(F,G)\to \Ext^*_{S(n,d)}(F(\k^n),G(\k^n))\simeq \Ext^*_{GL_n}(F(\k^n),G(\k^n))\;,$$
it suffices to prove that the first map is an isomorphism.
Assume that $F$ is $n$-resolved. Let $P_*$ be a resolution of $F$ by $n$-generated projectives. By lemma  \ref{lm-projngen}(i), the evaluation map $\Ext^*_{\P_\k}(F,G)\to \Ext^*_{S(n,d)}(F(\k^n),G(\k^n))$ is obtained by taking the homology of the chain map:
$\hom_{\P_\k}(P_*,G)\to \hom_{S(n,d)}(P_*(\k^n),G(\k^n))$. This latter map is an isomorphism by lemma \ref{lm-projngen}(ii) and we are done. The proof for $G$ $n$-coresolved is similar. 
\end{proof}

To make an effective use of theorem \ref{thm-nres-iso}, one needs ways of building $n$-resolved functors. Here come Troesch complexes into play.
\begin{proposition}[Examples of $n$-resolved functors]~\label{prop-ex-nresolved}
\begin{itemize}
\item[(i)] For all $r\ge 0$ and all $d\ge 0$, $\Gamma^{d\,(r)}$ is $p^r$-resolved.
\item[(ii)] If $F$ has degree $d$ then $F$ is $d$-resolved.
\item[(iii)] If $F_i$, $i=1,2$ is $n_i$ resolved then $F_1\otimes F_2$ is $(n_1+n_2)$-resolved.
\item[(iv)] If $F$ is $n$-resolved then $F$ is $(n+1)$-resolved.
\item[(v)] If $F$ is $n$-resolved and if $\ell$ is a positive integer, then $F(\k^\ell\otimes I)$ is $n\ell$-resolved.
\end{itemize} 
\end{proposition}
\begin{proof}
For (i), Troesch complexes provide injective coresolutions of $S^{d\,(r)}$ by injectives which are direct summands of $S^{p^rd}(\k^{p^r}\otimes I)$, hence $p^r$-cogenerated. Thus  $S^{d\,(r)}$ is $p^r$-coresolved, or equivalently $\Gamma^{d\,(r)}$ is $p^r$-resolved. The rest follows from lemma \ref{lm-prop-n-gen} by taking projective resolutions.
\end{proof}

\begin{remark}
In lemma \ref{lm-prop-n-gen} we proved that $F^{(r)}$ is $n$-generated if $F$ is $n$-generated. This is not the case for $n$-resolution. For example the identity functor $I=\Gamma^1$ is $1$-resolved, but example \ref{ex-calc} and theorem  \ref{thm-nres-iso} show that $I^{(1)}$ is not $1$-resolved. 
\end{remark}

\subsection{Towards a cohomological FFT and SFT for $GL_n$}\label{subsec-FFT-SFT}
Observe that in proposition \ref{prop-ex-nresolved}(i) the integer $n$ such that $\Gamma^{d\,(r)}$ is $n$-resolved only depends on $r$, and not on $d$. So from theorem \ref{thm-nres-iso} and section \ref{subsec-mult} we obtain new rational cohomology computations. 
For example corollary \ref{cor-alg} yields:
\begin{corollary}\label{cor-rat-alg}
Let $A^*$ be a family of strict polynomial functors, endowed with an algebra structure. For all $j\ge 0$, denote by $A^j_d$ the homogeneous degree $d$ part of $A^j$. Let $U$ be a finite dimensional vector space with trivial $GL_n$ action, let $\k^n$ be the standard representation of $GL_n$, and assume $n\ge p^r\dim U$. Then we have a trigraded algebra isomorphism
$$\bigoplus_{i,d,j}\Ext^i_{GL_n}(\Gamma^{d}(\hom_\k(U,\k^{n\,(r)})), A^j(\k^{n\,(r)}))\simeq \bigoplus_{i,d,j} A^j_d(U\otimes E_r)\;,$$
Similarly, if $C^*$ is endowed with a coalgebra structure, we have a trigraded algebra isomorphism 
$$\bigoplus_{i,d,j}\Ext^i_{GL_n}(C^j(\k^{n\,(r)}),S^{d}(U\otimes\k^{n\,(r)}))\simeq \bigoplus_{i,d,j} C^{j\,\sharp}_d(U\otimes E_r)\;.$$ 
\end{corollary}

\noindent
{\bf FFT and SFT for $GL_n$.} Now we turn to classical invariant theory.
Let us recall from \cite{DCP} the First and Second Fundamental Theorems for $GL_n$. Let $k,\ell$ be positive integers, let $r$ be a nonnegative integer and let  $V_r:=\k^{n\,(r)}$ be the $r$-th twist of the standard representation of $GL_n$. 

We define the contractions $(i|j)\in S^2(V_r^{\vee\,\oplus k}\oplus V_r^{\oplus\ell})$, for $1\le i\le k$, $1\le j\le \ell$ in the following way. Let $c_0\in V_r^\vee\otimes V_r\subset S^2(V_r^\vee\oplus V_r)$ denote the Trace (up to a scalar multiplication, it is the unique $GL_n$-invariant in $V_r^\vee\otimes V_r$ and in $S^2(V_r^\vee\oplus V_r)$). Let $\phi_{i,j}$ be the $GL_n$-equivariant map:
$$\begin{array}{cccc}
\phi_{i,j}:& V_r^\vee\oplus V_r &\to & V_r^{\vee\,\oplus k}\oplus V_r^{\oplus\ell}\\
& (f,x)&\mapsto & \underbrace{(0,\dots,0,f,0,\dots,0)}_{\text{$f$ in $i$-th position}},\underbrace{(0,\dots,0,x,0,\dots,0)}_{\text{$x$ in $j$-th position}}\;.
\end{array}
$$
Then the contraction $(i|j)$ is defined by $(i|j):= S^2(\phi_{ij})(c_0)$. 
\begin{remark}
The name `contraction' comes from the fact that $(i|j)$ may be equivalently defined as the homogeneous degree two polynomial (beware the duality):
$$\begin{array}{ccccl}(i,j): &V_r^{\oplus k}\oplus V_r^{\vee\,\oplus\ell}&\to &\k& \\
& (x_1,\dots,x_k,f_1,\dots,f_\ell)&\mapsto & f_j(x_i)&.
\end{array}
$$ 
\end{remark}

The  $(i|j)$ live in the invariant algebra $H^0(GL_n,S^*(V_r^{\vee\,\oplus k}\oplus V_r^{\oplus\ell}))$.
The First Fundamental Theorem \cite[Thm 3.1]{DCP} asserts that the $(i|j)$, $1\le i\le k$, $1\le j\le \ell$, actually generate this invariant algebra for $r=0$, hence for all $r\ge 0$ since Frobenius twists do not modify the invariants. The Second Fundamental Theorem \cite[Thm 3.4]{DCP} gives a minimal set of relations between the $(i|j)$ (these relations are given in terms of determinant formulas. In particular, if $n\ge \min\{k,\ell\}$, there are no relations between the $(i|j)$).

\medskip

\noindent
{\bf Cohomological FFT and SFT.} Now we are interested not only in the invariant algebra 
but more generally in the whole cohomology algebra $H^*(GL_n, S^*(V_r^{\vee\,\oplus k}\oplus V_r^{\oplus\ell}))$. 

To describe this algebra, we introduce `higher contractions' $(h|i|j)$, $0\le h<p^r$, $1\le i\le k$, $1\le j\le \ell$ in the following way.
First, by the exponential formula, one decomposes the representation $S^2(V_r^{\vee}\oplus V_r)$ as the direct sum:
$$S^2(V_r^{\vee}\oplus V_r)= S^2(V_r^{\vee})\oplus S^2(V_r)\oplus V_r^\vee\otimes V_r\;.$$ 
For weight reasons, the first two summands do not contribute to the $GL_n$-cohomology\footnote{Since the homotheties of $GL_n$ is a normal group-subscheme of $GL_n$, each rational $GL_n$-module $M$ splits as the direct sum of the $GL_n$-submodules $M_d$ of weight $d$, $d\in\Z$ ($M_d$ consists of the vectors $m$ of $M$ which are acted on by the homotheties via the formula $\lambda\Id\cdot m=\lambda^dm$). So the category of rational $GL_n$-modules splits as a direct sum of its full subcategories $\mathrm{Rat}_d$ whose objects are the rational $GL_n$-modules of weight $d$. Since the trivial representation $\k$ has weight $0$, the functor $\hom_{GL_n}(\k,-)$ is zero on $\mathrm{Rat}_d$, for $d\ne 0$. So its derived functors (i.e. $H^*(GL_n,-)$) vanish on $\mathrm{Rat}_d$ for $d\ne 0$.}, so the inclusion $V_r^\vee\otimes V_r\hookrightarrow S^2(V_r^{\vee}\oplus V_r)$ induces an isomorphism in cohomology. By \cite[I Chap 4, Lm 4.4]{Jantzen} there is an isomorphism $H^*(GL_n,V_r^\vee\otimes V_r)\simeq \Ext^*_{GL_n}(V_r,V_r)$. As a consequence, using the evaluation map $E_r\simeq\Ext^*_{\P_\k}(I^{(r)},I^{(r)})\to \Ext^*_{GL_n}(V_r,V_r)$, one finally obtains a graded map (which is an isomorphism if $n\ge p^r$):
$$E_r\to  H^*(GL_n,S^2(V_r^{\vee}\oplus V_r))\;.$$
Recall that $E_r$ is the graded vector space concentrated in degrees $2h$ for $0\le h<p^r$ and which equals $\k$ in these degrees. We denote by $c_h$ the image in $H^{2h}(GL_n,S^2(V_r^{\vee}\oplus V_r))$ of the canonical basis element of degree $2h$ of $E_r$ (in particular, $c_0$ is the Trace used before).
For $0\le h<p^r$, $1\le i\le k$, $1\le j\le \ell$, we define the higher contraction $(h|i|j)$ by the formula:
$$(h|i|j):= (S^2(\phi_{ij}))_*c_h \;\,\in\;\, H^{2h}(GL_n, S^2(V_r^{\vee\,\oplus k}\oplus V_r^{\oplus\ell}))\;.$$
In particular the $(0|i|j)$ are the usual contractions used in invariant theory.

\begin{theorem}[Partial Cohomological FFT and SFT]\label{cor-CFFT}
Let $k,\ell$ be positive integers. Let $r\ge 0$ and let $V_r:=\k^{n\,(r)}$ denote the $r$-th Frobenius twist of the standard representation of $GL_n$. Assume furthermore that $n\ge p^r\min\{k,\ell\}$. Then:
\begin{enumerate}
\item The cohomology algebra $H^*(GL_n, S^*(V_r^{\vee\,\oplus k}\oplus V_r^{\oplus\ell}))$ is generated by the contractions $(h|i|j)$, for $0\le h<p^r$, $1\le i\le k$, $1\le j\le \ell$.
\item There are no relations between the $(h|i|j)$.
\end{enumerate}  
\end{theorem}

\begin{proof}
We have to prove that $H^*(GL_n, S^*(V_r^{\vee\,\oplus k}\oplus V_r^{\oplus\ell}))$ is a symmetric algebra on the generators $(h|i|j)$.
The exponential formula yields an isomorphism of algebras (take the total degree on the right hand side) 
$$S^*(V_r^{\vee\,\oplus k}\oplus V_r^{\oplus\ell})\simeq S^*(V_r^{\vee\,\oplus k})\otimes S^*(V_r^{\oplus\ell})\;.$$
For weight reasons, the summands of the form $S^s(V_r^{\vee\,\oplus k})\otimes S^t(V_r^{\oplus\ell})$ for $s\ne t$ do not contribute to the $GL_n$-cohomology.
So \cite[I Chap 4, Lm 4.4]{Jantzen} yields an isomorphism of bigraded algebras:
\begin{align*}H^*(GL_n, S^*(V_r^{\vee\,\oplus k}\oplus V_r^{\oplus\ell}))&= \textstyle\bigoplus_d H^*(GL_n, S^d(V_r^{\vee\,\oplus k})\otimes S^d(V_r^{\oplus\ell}))\\& \simeq \textstyle\bigoplus_d \Ext^*_{GL_n}(\Gamma^d(V_r^{\vee\,\oplus k}), S^d(V_r^{\oplus\ell}))\;. \end{align*}
Thanks to corollary \ref{cor-rat-alg}, we know an isomorphism of bigraded algebras:
$$ \textstyle\bigoplus_d \Ext^*_{GL_n}(\Gamma^d(V_r^{\vee\,\oplus k}), S^d(V_r^{\oplus\ell}))\simeq \textstyle\bigoplus_d S^d(\k^k\otimes\k^\ell\otimes E_r)\;.$$
The elements of $\k^k\otimes\k^\ell\otimes E_r$ corresponds bijectively through these isomorphisms with $H^*(GL_n, S^2(V_r^{\vee\,\oplus k}\oplus V_r^{\oplus\ell}))$. Thus we have obtained that $H^*(GL_n, S^*(V_r^{\vee\,\oplus k}\oplus V_r^{\oplus\ell}))$ is a bigraded commutative algebra freely generated by its elements of bidegree $(*,2)$.

Now, for weight reasons, the exponential formula shows that the map 
$$\textstyle\sum_{i,j} S^2(\phi_{i,j}):\textstyle\bigoplus_{i,j} S^2(V_r^\vee\oplus V)\to S^2(V_r^{\vee\,\oplus k}\oplus V_r^{\oplus\ell})$$
induces an isomorphism in cohomology. So the $(h|i|j)$ form a basis of the vector space $H^*(GL_n, S^2(V_r^{\vee\,\oplus k}\oplus V_r^{\oplus\ell}))$, whence the result.
\end{proof}

\begin{question}[Conjectural Cohomological FFT for $GL_n$]\label{qn-CFFT}
We do not need any hypothesis on $n$ to build the $(h|i|j)$. The FFT asserts that the  $(0|i|j)$ always generate the invariant subalgebra $H^0(GL_n, S^*(V^{\vee\,\oplus k}\oplus V^{\oplus\ell}))$. This leads us to the following question: do the $(h|i|j)$  generate the cohomology algebra $H^*(GL_n, S^*(V^{\vee\,\oplus k}\oplus V^{\oplus\ell}))$ for all $n\ge 1$ ?
\end{question}

\begin{remark}\label{rk-SFT} Finding the relations between the $(h|i|j)$ seems more delicate. The SFT gives us the relations between the $(0|i|j)$. More generally, one could expect relations of the same kind between the $(h|i|j)$. There must also be some relations of the type $(h|i|j)=0$ for specific values of $h$ when $n$ is small, since some $c_h$ vanish when we are not in the semistable range. We don't know if these two kinds of relations are sufficient to generate the relations between the $(h|i|j)$ in general.
\end{remark}

\section{The twisting spectral sequence}\label{sec-untwist}

The following theorem generalizes the idea used by Cha{\l}upnik to compute extension groups in \cite{Chalupnik1}. It yields a spectral sequence which computes the extension groups $\Ext^{*}(F^{(r)}, G^{(r)})$ between twisted functors from extension groups between untwisted functors. So we call it the `twisting spectral sequence'. Its construction is a formal consequence of theorem \ref{thm-mult}. This spectral sequence is implicitly used in Cha{\l}upnik's proofs of \cite[Thm 4.3 and 4.4]{Chalupnik1}. It also appears in \cite[Prop 4.2.2]{FP}, but with a somewhat different identification of the second page (and without the compatibilities with cup products and twisting maps, which are very useful for applications).

\begin{theorem}[The twisting spectral sequence]\label{thm-sstwist}
Let $r$ be a nonnegative integer and let $F,G$ be strict polynomial functors. We consider (cf. section \ref{sec-background}) $G(E_r\otimes I)$ as a graded functor and we denote its  grading by the letter `$t$'.
There is a first quadrant cohomological spectral sequence, natural in $F,G$, 
$$E_2^{s,t}(F,G,r)=\Ext^s(F, G(E_r\otimes I))\Longrightarrow \Ext^{s+t}(F^{(r)}, G^{(r)})\;.$$
Moreover, these spectral sequences enjoy the following additional structure:
\begin{itemize}
\item[(i)] If $F_1,F_2,G_1,G_2$ are strict polynomial functors, there is a pairing of spectral sequences (see e.g. \cite[3.9]{Benson2} for a definition):
$$E^{*,*}(F_1,G_1,r)\otimes E^{*,*}(F_2,G_2,r)\to  E^{*,*}(F_1\otimes F_2,G_1\otimes G_2,r)\;.$$
On the second page and on the abutment, this pairing coincides with the product of extensions (defined in section \ref{subsec-mult}).
\item[(ii)] There are maps of spectral sequences (natural in $F,G$ and compatible with the pairing above):
$$E^{*,*}(F,G,r)\to E^{*,*}(F,G,r+1)\;.$$
These maps coincide with the split injections $\Ext^*(F,G(E_r\otimes I))\hookrightarrow \Ext^*(F,G(E_{r+1}\otimes I))$ on the second page, and with the twisting maps $\Ext^*(F^{(r)},G^{(r)})\to \Ext^*(F^{(r+1)},G^{(r+1)})$ on the abutment.
\end{itemize}
\end{theorem}

\begin{remark}
Assume that $C^*$, resp $A^*$, is a family of strict polynomial functors endowed with a graded coalgebra (resp.  algebra) structure. Then $\bigoplus_{s,i,j}\Ext^s(C^{i\,(r)},A^{j\,(r)})$ is a trigraded algebra (whose multiplication sends $c_1\otimes c_2$ to $(m_A^{(r)})_*(\Delta_C^{(r)})^*(c_1\cup c_2)$, where $\Delta_C$, resp. $m_A$, denotes the comultiplication in $C^*$, resp. multiplication in $A^*$ ). Similarly, $\bigoplus_{s,t,i,j}\Ext^s(C^i, A^j(E_r\otimes I))$ is a quadrigraded algebra (here $t$ denotes the grading of the functors $A^j(E_r\otimes I)$).
Using the pairing of theorem \ref{thm-sstwist}(i) and naturality, one obtains a spectral sequence of quadrigraded algebras:
$$\bigoplus_{s,t,i,j} E_2^{s,t}(C^i,A^j,r) 
\Longrightarrow 
\bigoplus_{s+t,i,j}\Ext^{s+t}(C^{i\,(r)},A^{j\,(r)})\;.$$
\end{remark}

\begin{proof}[Proof of theorem \ref{thm-sstwist}] The twisting spectral sequence is just an hypercohomology spectral sequence, and we identify the second page by theorem \ref{thm-mult}. 

Let us give the details. Let $F,G\in\P_\k$. Let $J_G^*$ be an injective coresolution of $G$. So $J_G^{*\,(r)}$ is a coresolution of $G^{(r)}$ and we denote by $J_{G,r}^{**}$ a Cartan-Eilenberg coresolution of $J_G^{*\,(r)}$. The totalization of $J_{G,r}^{**}$ is an injective coresolution of $G^{(r)}$, so that $\Ext^*(F^{(r)},G^{(r)})$ equals  the homology of the totalization of $\hom(F^{(r)},J_{G,r}^{**})$.
The spectral sequence $E(F,G,r)$ associated to this bicomplex (which computes the homology along the columns first, then along the rows, etc.) converges to $\Ext^*(F^{(r)},G^{(r)})$. Tensor products induce pairings of bicomplexes 
$$\hom(F_1^{(r)},J_{G_1,r}^{i,j})\otimes \hom(F_2^{(r)},J_{G_2,r}^{k,\ell})\to \hom((F_1\otimes F_2)^{(r)},J_{G_1\otimes G_2,r}^{i+k,j+\ell})$$
whence pairings of spectral sequences $E(F_1,G_1,r)\otimes E(F_2,G_2,r)\to E(F_1\otimes F_2,G_1\otimes G_2,r)$. Finally, the map $E(F,G,r)\to E(F,G,r+1)$ comes from the composite:
$$\hom(F^{(r)},J_{G,r}^{**})\simeq \hom(F^{(r+1)},J_{G,r}^{**\,(1)})\to \hom(F^{(r+1)},J_{G,r+1}^{**})\;. $$

Now we identify the second page. The first page of $E(F,G,r)$ (together with the first differential $d_1:E_1^{s,t}(F,G)\to E_1^{s+1,t}(F,G)$) equals the $t$-graded complex  $\Ext^t(F^{(r)},J_G^{*\,(r)})$. By theorem \ref{thm-mult}, this complex is isomorphic (in a $t$-graded way, and compatibly with cup products and Frobenius twist) to  $\hom(F, J_G^*(E_r\otimes I))$. But $J_G^k(E_r\otimes I)$ is injective in each degree, so $J_G^{*}(E_r\otimes I)$ is an injective $t$-graded coresolution of $G(E_r\otimes I)$. Thus the second page identifies with $\Ext^s(F, G(E_r\otimes I))$. Moreover the pairing of spectral sequences is given by cup products on the first pages of the form $\Ext^t(F^{(r)},J_G^{*\,(r)})$, so by tensor product after the identification of the first page with $\hom(F, J_G^*(E_r\otimes I))$, so by cup products on the second page. Finally, the map induced by precomposition by the Frobenius twist  is given at the first page level by the split injection $\hom(F, J_G^*(E_r\otimes I))\hookrightarrow \hom(F, J_G^*(E_{r+1}\otimes I))$, whence its description at the second page level.
\end{proof}

Of course, starting with a projective resolution of $F$ in the proof, one obtains another first quadrant cohomological spectral sequence:
$${E'}_2^{s,t}(F,G)=\Ext^s(F(E_r\otimes I), G)\Longrightarrow \Ext^{s+t}(F^{(r)}, G^{(r)})\;.$$
At first sight, the second page of $E'(F,G,r)$ might seem quite different from the second page of $E(F,G,r)$. But the following proposition shows it is not. Thus, we don't have any advantage in considering $E'$ (and that's why we do not use it later).
\begin{proposition}
Let $F,G$ be strict polynomial functors. For all finite dimensional graded $\k$-vector space $V$, there is a bigraded isomorphism, natural in $F,G$, and $V$ (recall that `$~^\vee$' denotes $\k$-linear duality)
$$\Ext^*(F(V\otimes I),G)\simeq \Ext^*(F,G(V^\vee\otimes I))\;.$$
\end{proposition}
\begin{proof}
It suffices to build an isomorphism natural in $P,J$ and $V$ $$\hom(P(V\otimes I),J)\simeq\hom(P,J(V^\vee\otimes I))\;,$$  when $P=P_U=\Gamma^{d}(\hom_\k(U,-))$ and $J=J_W=S^{d}(\hom_\k(W,-))$ (where $U$ and $W$ are finite dimensional $\k$-vector spaces). The general result then follows by taking resolutions. Now, $P_U(V\otimes I)=P_{U\otimes V^\vee}$ and $J_W(V^\vee\otimes I)=J_{W\otimes V}$ , so the Yoneda lemma yields an isomorphism, natural in $P_U,J_W,V$:
\begin{align*}\hom(P_U(V\otimes I), J_W)&\simeq  J_W(U\otimes V^\vee)= \\&
(P_U)^\sharp (V^\vee\otimes W^\vee)\simeq \hom(P_U, J_W(V^\vee\otimes I))\;.
\end{align*}
Whence the result.
\end{proof}

\begin{remark}
In fact, with more work, one can show that not only the second page and the abutment of the two spectral sequences $E_i(F,G,r)$ and $E'_i(F,G,r)$ are the same, but more generally there is a spectral sequence $E''_i(F,G,r)$ and maps of spectral sequences $E_i(F,G,r)\to E_i''(F,G,r)$ and $E_i'(F,G,r)\to E_i''(F,G,r)$ which yield isomorphisms between the second pages (and so also between the $i$-th pages for $i\ge 2$).
\end{remark}

The remainder of the section is devoted to applications of the twisting spectral sequence.

\subsection{Effect of precomposition by Frobenius twists}\label{subsec-precomp}

The twisting spectral sequence is a convenient way to study the twisting map 
$$\mathrm{Fr}_1:\Ext^*(F^{(r)},G^{(r)})\to \Ext^*(F^{(r+1)},G^{(r+1)})\;,$$
induced by precomposition by $I^{(1)}$. We already know that:
\begin{enumerate}
\item This map is injective, at least when $\k=\mathbb{F}_p$. This is the `twist injectivity'. It results from a theorem of Andersen \cite[Part II, prop 10.14]{Jantzen} which proves that for $G$ reductive, the map induced by \emph{postcomposition} by Frobenius twists $H^*(G,V)\to H^*(G,V^{(1)})$ is injective. Since $\k=\mathbb{F}_p$, we have \cite[proof of Thm 4.10]{FS} $I^{(1)}\circ F \simeq F\circ I^{(1)}$, so this result for postcomposition by Frobenius twists is actually also valid for precomposition.
\item This map is an isomorphism in low degrees. This is the `strong twist stability' of \cite[Cor 4.10]{FFSS}, and for $\k=\mathbb{F}_p$ it also results from \cite[Th (6.6)]{CPSVdK}.
\end{enumerate}

As a first application of the twisting spectral sequence, we get a slight improvement of the `strong twist stability'.
\begin{corollary}[{compare \cite[Cor 4.10]{FFSS}}]
Let $F,G$ be strict polynomial functors, let $j,r$ be nonnegative integers. The twisting map
$$\Ext^s(F^{(r)},G^{(j+r)})\to \Ext^s(F^{(r+1)},G^{(j+r+1)})$$
is an isomorphism for $0\le s< 2p^{r+j}$, and is injective in degree $s=2p^{r+j}$.
\end{corollary}
\begin{proof}
The two graded vector spaces $E_r^{(j)}$ and $E_{r+1}^{(j)}$ are equal in degrees $t<2p^{r+j}$ (recall that $-^{(j)}$ acts as a homothety of coefficient $p^j$ on the degrees). So the injective map $G^{(j)}(E_r\otimes I)\to G^{(j)}(E_{r+1}\otimes I)$ is an isomorphism in degrees $t<2p^{r+j}$. In particular, the map of spectral sequences $E^{s,t}_2(F,G^{(j)},r)\to E^{s,t}_2(F,G^{(j)},r+1)$ is an isomorphism in total degree $(s+t)<2p^{r+j}$. Now the result follows by induction (on the page number of the spectral sequence) from the following elementary result. Let $f^*:C^*\to D^*$ be a morphism of cochain complexes, which is an isomorphism in degrees $k<\ell$ and is injective in degree $\ell$. Then $H^k(f):H^k(C)\to H^k(D)$ is an isomorphism in degrees $k<\ell$ and is injective in degree $\ell$.
\end{proof}

We don't know any proof in $\P_\k$ of the `twist injectivity' phenomenon. The following statement may be a step towards such a proof (see corollary \ref{cor-chal} and section \ref{sec-collapse} for many results on collapsing).
\begin{corollary}\label{cor-twistinjcoll}
Let $F,G$ be strict polynomial functors and let $r\ge 0$. Assume that the spectral sequence $E(F,G,r+1)$ collapses at the second page. Then $E(F,G,r)$ also collapses at the second page and the twisting map is injective:
$$\Ext^*(F^{(r)},G^{(r)})\hookrightarrow \Ext^*(F^{(r+1)},G^{(r+1)})\;.$$ 
\end{corollary}
\begin{proof}
We have an injection $E_2^{s,t}(F,G,r)\hookrightarrow E_2^{s,t}(F,G,r+1)$. Since $E(F,G,r+1)$ collapses at the second page, so does $E(F,G,r)$ (it is a sub-spectral sequence) and we get an injection $E_\infty^{s,t}(F,G,r)\hookrightarrow E_\infty^{s,t}(F,G,r+1)$. So precomposition by the Frobenius twist induces a an injective map
$\mathrm{Gr}(\Ext^*(F^{(r)},G^{(r)}))\hookrightarrow \mathrm{Gr}(\Ext^*(F^{(r+1)},G^{(r+1)}))$ (`$\mathrm{Gr}$' denotes the graded object coming from the filtration on the abutment), hence an injective map $\Ext^*(F^{(r)},G^{(r)})\hookrightarrow \Ext^*(F^{(r+1)},G^{(r+1)})$.
\end{proof}

Finally, the twisting spectral sequence also gives information about numerical invariants associated to $\Ext^*(F^{(r)},G^{(r)})$.  For example, corollary \ref{cor-num}(ii) enables to control the growth of the total dimension of $\Ext^*(F^{(r)},G^{(r)})$ viewed as a function of $r$ (by \cite{Totaro} and \cite[Th 2.10]{FS}, we know that this total dimension is finite). 

\begin{corollary}\label{cor-num} Let $r$ be a positive integer and let $F,G$ be strict polynomial functors. 
\begin{itemize}
\item[(i)] The graded vector spaces $\Ext^*(F^{(r)}, G^{(r)})$, $\Ext^*(F,G(E_r\otimes I))$ (with the total grading) and $\Ext^*(F,G(\k^{p^r}\otimes I))$ have the same Euler characteristic.
\item[(ii)] The total dimension of $\Ext^*(F^{(r)}, G^{(r)})$ is less or equal to the total dimension of $\Ext^*(F,G(\k^{p^r}\otimes I))$. Equality holds iff the twisting spectral sequence collapses at the second page (i.e. $E^{**}_2(F,G)= E^{**}_\infty(F,G)$).
\end{itemize}
\begin{proof}
The graded functor $G(E_r\otimes I)$ is concentrated in even degrees since $E_r$ is. Moreover, $E_r=\k^{p^r}$ as ungraded vector spaces, so $\Ext^*(F,G(E_r\otimes I))$ and $\Ext^*(F,G(\k^{p^r}\otimes I))$ have the same Euler characteristic. Now if $C$ is a complex, the Euler characteristic of $C^*$ equals the Euler characteristic of $H^*(C)$. Applying this to $E^{**}_r(F,G)$ we get (i). For (ii), observe that the total dimension of $E_{r+1}^{* *}(F,G)$ equals the total dimension of $E_{r}^{* *}(F,G)$ minus twice the total rank of $d_r^{* *}$.
\end{proof}
\end{corollary}

\subsection{Detecting extensions between twisted functors}

Consider the edge homomorphism $\theta(F,G,r)$ of the twisting spectral sequence:
$$\Ext^*(F^{(r)},G^{(r)})\twoheadrightarrow E^{0,*}_{\infty}(F,G)\subset E^{0,*}_{2}(F,G)=\hom(F,G(E_r\otimes I))\;.$$
Its target is easy to compute (this is invariant theory). So we want to use it as a way to detect extensions in $\Ext^*(F^{(r)},G^{(r)})$ (or even to compute these extension groups). To this purpose, we first state some elementary properties of the edge homomorphism.
\begin{lemma}\label{lm-std-prop}
\begin{itemize}
\item[(i)] If $F$ is arbitrary and $G=S^\mu$, the edge homomorphism is an isomorphism. Furthermore, it coincides with the isomorphism of theorem \ref{thm-mult} (whence the notation $\theta(F,G,r)$).
\item[(ii)] The edge homomorphism is compatible with products:
$$\theta(F_1\otimes F_2,G_1\otimes G_2 ,r)(c_1\cup c_2)=\theta(F_1,G_1,r)(c_1)\otimes\theta(F_2,G_2,r)(c_2)\;.$$
\item[(iii)] The edge homomorphisms $\theta(F,G,r)$ and $\theta(F,G,r+1)$ fit into a commutative diagram:
$$\xymatrix{
\Ext^*(F^{(r)},G^{(r)})\ar[d]\ar[rr]^-{\theta(F,G,r)}_-{\simeq} &&\hom(F,G(E_r\otimes I))\ar@{^{(}->}[d]\\
\Ext^*(F^{(r+1)},G^{(r+1)})\ar[rr]^-{\theta(F,G,r+1)}_-{\simeq}&&\hom(F,G(E_{r+1}\otimes I))\;\;.
}$$
\end{itemize} 
\end{lemma}
\begin{proof} 
For (i), remember the construction of the twisting spectral sequence. Since $S^\mu$ is injective, we can take an injective bicomplex $J^{**}$ which is zero in bidegree $(s,t)$ as soon as $t>0$. Thus, before identifying the second page, the edge isomorphism is simply the identity map: 
$$\Ext^*(F^{(r)},S^{\mu\,(r)})=H^*_{d_0}(\hom(F^{(r)},J^{* 0}))= H^0_{d_1}(H^*_{d_0}(\hom(F,J^{**}))\,.$$
Since we use the isomorphism $\theta(F,S^\mu,r)$ to identify the second page, the edge homomorphism coincides with $\theta(F,S^\mu,r)$ after identification.

To prove (ii) and (iii), observe that these properties are already known when $G=S^\mu$ (by (i)). Now $G$ embeds in an injective $J$ which is a finite direct sum of $S^\mu$. And we have a commutative diagram:
$$\xymatrix{
\Ext^*(F^{(r)},G^{(r)})\ar[d]\ar[rr]_-{\theta(F,G,r)}&&\hom(F,G(E_r\otimes I))\ar@{^{(}->}[d]\\
\Ext^*(F^{(r)},J^{(r)})\ar[rr]_-{\theta(F,J,r)}^-{\simeq}&&\hom(F,J(E_r\otimes I))\;.
}$$
Since the vertical arrow on the right is injective, one can read cup products as well as the effect of the Frobenius twist in $\hom(F,J(E_r\otimes I))$ where we know (ii) and (iii) are valid. This proves (ii) and (iii). 
\end{proof}

Now assume that the twisting spectral sequence collapses at the second page. Then the edge homomorphism is surjective (so it detects many extensions).

\begin{corollary}\label{cor-chal}
Let $r$ be a positive integer and let $F,G$ be strict polynomial functors. 
Assume that $\Ext^s(F, G(E_r\otimes I))=0$ for $s$ odd. Then the twisting spectral sequence collapses at the second page. So one has an isomorphism (compatible with cup products and Frobenius twisting).
$$ \mathrm{Gr}(\Ext^{s+t}(F^{(r)},G^{(r)}))\simeq E_\infty^{s+t}(F,G) \simeq \Ext^s(F,G(E_r\otimes I))\;.$$
In particular the edge homomorphism is a surjection (compatible with cup products and Frobenius twisting):
$$\Ext^*(F^{(r)},G^{(r)})\twoheadrightarrow \Ext^0(F,G(E_r\otimes I))\;.$$
Its kernel is isomorphic (non naturally in $F,G$) to $\Ext^{>0}(F,G(E_r\otimes I))$.
\end{corollary}
\begin{proof}
The graded functor $G(E_r\otimes I)$ is concentrated in even degrees since $E_r$ is. Thus $E_2^{s,t}(F,G,r)$ is zero unless $s$ and $t$ are even. So the second page is concentrated in even total degree. The differentials of the spectral sequence raise the total degree by one, so they must vanish.  
\end{proof}

\begin{remark}\label{rk-collapse-chal}
Corollary \ref{cor-chal} is a generalization of the main theorem of \cite{Chalupnik1}. Indeed, assume that the stronger $\Ext$-condition holds: $\Ext^s(F, G(E_r\otimes I))=0$ for $s>0$. Then $E_2^{s,t}(F,G,r)$ is concentrated in bidegrees $(0,t)$ so the edge homomorphism is actually an isomorphism (compatible with cup products and Frobenius twisting), compare \cite[Th 4.4]{Chalupnik1}:
$$\Ext^{*}(F^{(r)},G^{(r)})\simeq \hom(F,G(E_r\otimes I))\;.$$

As examples of pairs of functors satisfying this strong $\Ext$-condition, one can take: (i) $F$ is a projective functor and $G$ is arbitrary, (ii) $F$ is arbitrary and $G$ is injective, (iii) $F^\sharp$ and $G$ are Schur functors (this case gives \cite[Th 6.1]{Chalupnik1}), or more generally (iv) $F^\sharp$ and $G$ have a Schur filtration (i.e. a filtration whose associated graded object is a direct sum of Schur functors). Many $\Ext^*$-computations between twisted functors known in the literature amount to one of these cases.

As a pair satisfying the vanishing in odd degrees but not the strong $\Ext$-condition, one can take for example $F=\Gamma^{d\,(1)}$ and $G=I^{(1)}\otimes S^{dp-p}$.
\end{remark}

\section{The collapsing conjecture}\label{sec-collapse}

As observed in remark \ref{rk-collapse-chal}, many explicit computations in the literature amount to the collapsing of the twisting spectral sequence at the second page.
In fact, when one looks carefully at the results of \cite{FS,FFSS,Chalupnik1,Chalupnik2}, it seems that the total dimension of $\Ext^{*}(F^{(r)},G^{(r)})$ always equals the total dimension of $\Ext^*(F,G(\k^{p^r}\otimes I))$.
Thus (cf. corollary \ref{cor-num}(ii)), it seems that the twisting spectral sequence always collapses at the second page, even when there is no obvious lacunary reason for it.
This leads us to the following conjecture.

\begin{conjecture}\label{conj-collapse}
For all $r\ge 0$, and all $F,G$,  the twisting spectral sequence $E(F,G,r)$ of theorem \ref{thm-sstwist} collapses at the second page.
\end{conjecture}

\begin{remark}\label{rk-collapse} One can reduce conjecture \ref{conj-collapse} to slightly less general statement. For example, by induction, one can reduce the conjecture to the case $r=1$. We also know that the collapsing of $E(F,G,r+1)$ implies the collapsing of $E(F,G,r)$ by corollary \ref{cor-twistinjcoll}. So, to prove the conjecture, it suffices to find for each pair $F,G$ an integer $r$ such that $E(F,G,r)$ collapses at the second page.
\end{remark}

The purpose of this section is to make a step towards conjecture \ref{conj-collapse}. Namely, we prove that the spectral sequence $E(F,G,r)$ collapses at the second page for all $r\ge 1$ for a big family of pairs $(F,G)$, which contains all the pairs studied in  \cite{FS,FFSS,Chalupnik1,Chalupnik2}, and many others. This collapsing is stated in theorem \ref{thm-collapse}, which is the main result of the section. 

Our general approach to conjecture \ref{conj-collapse} and to theorem \ref{thm-collapse} is based once again on Troesch complexes, and may be described as follows. Recall the way we constructed the twisting spectral sequence $E(F,G,r)$. We first took an injective coresolution $J_G^*$ of $G$. So $J_G^{*\,(r)}$ is a exact coresolution of $G^{(r)}$ by twisted symmetric powers. The twisting spectral sequence $E(F,G,r)$ is constructed as the spectral sequence associated to the bicomplex $\hom(F^{(r)}, J_{G,r}^{*,*})$, where $J_{G,r}^{*,*}$ is a Cartan-Eilenberg injective coresolution of $J_G^{*\,(r)}$.
The idea to prove our collapsing result is to construct a bicomplex $J_{G,r}^{*,*}$, whose columns are Troesch coresolutions (i.e. of the form $T(S^\mu,r)^*$). Indeed, we have seen in lemma \ref{lm-CFSr} that the complexes $\hom(F^{(r)},T(S^\mu,r))$ are zero in odd degrees. So, with such a bicomplex $J_{G,r}^{*,*}$ in hand, the odd degree rows of the bicomplex $\hom(F^{(r)},J_{G,r}^{*,*})$ are zero, so that the differentials $d_i$ of the associated spectral sequence are zero for $i\ge 2$.

To construct such a bicomplex, we want to apply a functor `$T(-,r)$' to the coresolution $J_G$. We have already defined the values of $T(-,r)$ on symmetric powers in definition \ref{def-TSmur}. But it is not clear how to define \emph{in a functorial way} the values of $T(-,r)$ on morphisms between symmetric powers. To bypass this difficulty, we introduce in section \ref{subsec-tcc} `twist compatible categories' $\mathcal{T}_r\mathcal{P}_\k$ which extend to the case of arbitrary $r$ the twist compatible categories already used in \cite{TouzeUniv} for $r=1$. These categories are (not full) subcategories of $\P_{\k}$ containing symmetric powers. Since $\mathcal{T}_r\mathcal{P}_\k$ has less morphisms than $\mathcal{P}_\k$, it is easier to make $T(-,r)$ into a functor on this category (cf. proposition \ref{prop-T}). In section \ref{subsec-answer}, we use the functor $T(-,r)$ to fulfill our plans. Namely, if $G$ has a `twist compatible coresolution' (i.e. an injective coresolution of $G$ whose differentials are morphisms of the twist compatible category), our idea works by taking $J_{G,r}^{*,*}=T(J_G^*,r)^*$. This proves theorem \ref{thm-collapse}, and we observe in proposition \ref{prop-ex} that many interesting functors admit such twist compatible coresolutions. Finally, we study in section \ref{subsec-comb} to what extent our approach could be used to prove conjecture \ref{conj-collapse} for all pairs $(F,G)$. In particular, we propose an elementary combinatorial problem whose solution would prove the conjecture.

\subsection{Twist compatible categories}\label{subsec-tcc}

In this section, we choose a functorial construction of the direct sum in $\P_\k$. (for example, $F\oplus G$ \emph{equals} the functor which sends a vector space $V$ to the set of couples $(f,g)$, with $f\in F(V)$ and $g\in G(V)$).

\begin{definition}[{compare \cite[Def. 3.4]{TouzeUniv}}]\label{def-r-twist} Let $r$ be a positive integer,
let $(\lambda^i)$ and $(\mu^j)$ be two finite families of tuples of nonnegative integers, and let $f\in \hom(\bigoplus_i S^{\lambda^i}, \bigoplus_j S^{\mu^j})$. We say that $f$ is $r$-twist compatible if there exists a map $\overline{f_r}$ such that the following diagram commutes:
$$\xymatrix{
\bigoplus_i S^{\lambda^i}(S^{p^r})\ar@{->>}[d]\ar@{->}[rr]^-{f(S^{p^r})}&& \bigoplus_j S^{\mu^j}(S^{p^r})\ar@{->>}[d]\\
\bigoplus_i S^{p^r\lambda^i}\ar@{->}[rr]^-{\overline{f_r}}&& \bigoplus_j S^{p^r\mu^j}
}
$$
where the vertical epimorphisms are induced by multiplications $S^n(S^{p^r})\twoheadrightarrow S^{np^r}$ (it is the unique map in $\hom(S^n(S^{p^r}),S^{np^r})$ up to a scalar constant).
\end{definition}

\begin{remark}
In general, not all isomorphisms $\bigoplus_i S^{\lambda^i}(S^{p^r})\simeq \bigoplus_i S^{\lambda^i}(S^{p^r})$ are $r$-twist compatible. That's why it is important to choose a strict definition of the direct sum : with a definition `up to isomorphism', definition \ref{def-r-twist} would be inconsistent.
\end{remark}

If $\overline{f_r}$ exists, it is unique by surjectivity of the vertical arrows. Observe that $r$-twist compatible maps form a subvector space of the $\k$ vector space $\hom(\bigoplus_i S^{\lambda^i}, \bigoplus_j S^{\mu^j})$, and that the composite of two $r$-twist compatible maps is $r$-twist compatible. Examples of $r$-twist compatible maps are the multiplications $S^i\otimes S^j\to S^{i+j}$ or the permutations $S^i\otimes S^j\simeq S^j\otimes S^i$.
So, the category $\overline{\mathcal{T}_r\mathcal{P}_\k}$ whose objects are the finite direct sums of $S^\mu$ and whose maps are $r$-twist compatible is a $\k$-linear subcategory of $\P_\k$, stable by direct sums.

For combinatorial reasons, we would like our category to be stable not only by direct sums but also by tensor products. So we have to enlarge it a bit. To this purpose we introduce `iterated symmetric tensors'. A $0$-iterated symmetric tensor is just a functor of the form $S^\mu$. For $n\ge 1$, an $n$-iterated symmetric tensor is a functor $F$ of the form $F:=\bigotimes_{i=1}^k\bigoplus_{j=1}^\ell S_{i,j}$ where $S_{i,j}$ is an $(n-1)$-iterated symmetric tensor.
If $F$ is an $n$-iterated symmetric tensor, distributivity of tensor products with respect to direct sums yields a canonical isomorphism $\xi_F :F\simeq F_0$ where $F_0$ is a direct sum of $0$-iterated symmetric tensors. 

\begin{definition}
The $r$-th twist compatible category $\mathcal{T}_r\mathcal{P}_\k$ is the (not full) subcategory of $\P_\k$ whose objects are iterated symmetric tensors and whose morphisms are the $f:F\to G$ such that $f_0:=\xi_G\circ f\circ \xi_F^{-1}$ are $r$-twist compatible.
\end{definition}

One checks as in \cite[Lm 3.10]{TouzeUniv} that $\mathcal{T}_r\mathcal{P}_\k$ is stable by direct sums and tensor products. Moreover, we have a functor (actually an equivalence of categories), which enables us to extend any functor with source $\overline{\mathcal{T}_r\mathcal{P}_\k}$ into a functor with source $\mathcal{T}_r\mathcal{P}_\k$:
$$\begin{array}{ccc}\mathcal{T}_r\mathcal{P}_\k &\to &\overline{\mathcal{T}_r\mathcal{P}_\k}\\
F&\mapsto & F_0\\
F\xrightarrow[]{f} G& \mapsto & F_0\xrightarrow[]{f_0} G_0
\end{array}
$$

The interest of $\mathcal{T}_r\mathcal{P}_\k$ (or $\overline{\mathcal{T}_r\mathcal{P}_\k}$) lies in the following observation. Let $f\in\hom(\bigoplus_i S^{\lambda^i}, \bigoplus_j S^{\mu^j})$ be an $r$-twist compatible map. Since $f^{(r)}$ is just the restriction of $f(S^{p^r})$ to the subfunctor $\bigoplus_i S^{\lambda^i\,(r)}$, we have a commutative diagram (the vertical arrows are the canonical inclusions):
$$\xymatrix{
\bigoplus_i S^{\lambda^i\,(r)}\ar@{^{(}->}[d]\ar@{->}[rr]^-{f^{(r)}}&& \bigoplus_j S^{\mu^j\,(r)}\ar@{^{(}->}[d]\\
\bigoplus_i S^{p^r\lambda^i}\ar@{->}[rr]^-{\overline{f_r}}&& \bigoplus_j S^{p^r\mu^j}
}$$
Thus $\overline{f_r}$ determines way to choose a lifting of $f^{(r)}$ to $\bigoplus_i S^{p^r\lambda^i}\to \bigoplus_j S^{p^r\mu^j}$, in a way  which respects composition (there are often many liftings of $f^{(r)}$ available, so making a `good' choice is not trivial). 

Now let $T(\oplus_i S^{\lambda^i},r)^*$ be the injective coresolution of $\oplus_i S^{\lambda^i\,(r)}$ obtained by contracting the $p$-complex $\bigoplus_j B^*_{p^r\lambda^i}(r)$ (that is, the complex $T(\oplus_i S^{\lambda^i},r)^*=\oplus_i T( S^{\lambda^i},r)^* $ is the one used in section \ref{subsec-compute}). If $f$ is a $r$-twist compatible map define $T(f,r)^*:= \overline{f}(\sha_r\otimes I)_{[1]}$. By proposition \ref{prop-res-2}, this defines a functor from $\overline{\mathcal{T}_r\mathcal{P}_\k}$ (or equivalently from $\mathcal{T}_r\mathcal{P}_\k$) to cochain complexes. To sum up, we have proved the following result.

\begin{proposition}\label{prop-T}
Let $r$ be a nonnegative integer. The complexes $T(S^\mu,r)^*$ of section \ref{subsec-compute} define an additive functor $T(-,r)^*$ from the $r$-twist compatible category $\mathcal{T}_r\mathcal{P}_\k$ to the category of cochain complexes, which sends an object $F$ to an injective coresolution of $F^{(r)}$ and a morphism $f:F\to G$ to a chain map $T(f,r)^*:T(F,r)^*\to T(G,r)^*$ which lifts $f^{(r)}$. 
\end{proposition}

\begin{remark}
For $r=1$, we have $T(-,1)^*=-_{[1]}\circ T$ where $T$ is the functor of \cite[Prop 3.13]{TouzeUniv} ($T$ has values in $p$-complexes). Contrarily to \cite[Prop 3.13]{TouzeUniv}, we do not assert any compatibility of $T(-,r)$ with tensor products in proposition \ref{prop-T}. Indeed for $r\ge 2$, the isomorphism $B(r)_d^*(V\oplus W)\simeq \bigoplus_{i+j=d}B(r)_i^*(V)\otimes B(r)_j^*(W)$ induced by the exponential formula does not commute with the $p$-differentials, so it is not clear how to get a compatibility result with tensor products (this problem already appears in remark \ref{rk-iso}).
\end{remark}

In section \ref{subsec-compute}, we computed the complexes $\hom(F^{(r)},T(S^\mu,r)^*)$. Now we can consider $T(-,r)$ as a functor, and the analysis performed in section \ref{subsec-compute} yields the following result.

\begin{proposition}\label{prop-analysis}
Let $F\in \P_\k$ and let $r$ be a nonnegative integer. For all $J\in\mathcal{T}_r\mathcal{P}_\k$, the cochain complex  
$\hom(F^{(r)},T(J,r)^*) $ is zero in odd degrees (so its differential is trivial). Moreover, 
there is a graded isomorphism, natural with respect to morphisms $f:F'\to F$ in $\P_\k$  and $g:J\to J'$ in $\mathcal{T}_r\mathcal{P}_\k$:
$$\hom(F^{(r)},T(J,r)^*) \simeq \hom(F,J(E_r\otimes I))\;. $$
\end{proposition}
\begin{proof} The cancellation property is given in lemma \ref{lm-CFSr}. The naturality is a particular case of lemma \ref{lm-nat}.
\end{proof}

\subsection{A partial answer to conjecture \ref{conj-collapse}}\label{subsec-answer}

\begin{definition}
Let $F\in \P_\k$ and let $r$ be a positive integer. An $r$-twist compatible coresolution of $F$ is a cochain complex $J_F^*$ in $\mathcal{T}_r\mathcal{P}_\k$, which is a coresolution of $F$\footnote{The objects of $\mathcal{T}_r\mathcal{P}_\k$ are injectives of $\P_{\k}$. Thus, $r$-twist compatible coresolutions of $F$ are actually injective coresolutions of $F$.}. 
\end{definition}

\begin{proposition}[examples]\label{prop-ex}
Let $r$ be a positive integer.

\begin{itemize}
\item[(i)] The symmetric powers $S^n$ admit $r$-twist compatible coresolutions.
\item[(ii)] The kernels of the reduced bar complex $\overline{B}(S^*)$ (in particular the exterior powers $\Lambda^n$) admit $r$-twist compatible coresolutions. 
\item[(iii)] The kernels of the reduced bar complex $\overline{B}(\Lambda^*)$ (in particular the divided powers $\Gamma^n$) admit  $r$-twist compatible coresolutions.
\item[(iv)] If $F$ and $G$ admit $r$-twist compatible coresolutions, so does $F\otimes G$.
\end{itemize}
\end{proposition}
\begin{proof}
For (i), take  $S^n$ as a resolution of $S^n$. For (ii), we know  \cite[Lm 3.19]{TouzeUniv} that $\overline{B}(S^*)$ has homology $\Lambda^n$. Moreover, the differential of $\overline{B}(S^*)$ is $r$-twist compatible, since it is build by taking linear combinations of tensor products of permutations $S^i\otimes S^j\simeq S^j\otimes S^i$ and multiplications $S^i\otimes S^j\to S^{i+j}$ which are $r$-twist compatible. So, the homogeneous part of polynomial degree $n$ of $\overline{B}(S^*)$ yields a complex $J_{\Lambda^n}^*$ in $\mathcal{T}_r\mathcal{P}_\k$ which looks like:
$$\underbrace{\otimes^n}_{J_{\Lambda^n}^0}\to \underbrace{\textstyle\bigoplus_{i+j=n-1} (\otimes^i)\otimes S^2\otimes (\otimes^j)}_{J_{\Lambda^n}^1} \to \dots\to \underbrace{S^n}_{J_{\Lambda^n}^{n-1}}\;.$$ 
The cohomology of $J_{\Lambda^n}^*$ equals $\Lambda^n$ concentrated in degree $0$, so $\Lambda^n$ admits $J_{\Lambda^n}^*$ as an $r$-twist compatible coresolution. If $K$ is degree $n$ homogeneous strict polynomial functor which is a kernel of a differential of $\overline{B}(S^*)$, then $K$ admits a truncation of $J_{\Lambda^n}^*$ as an $r$-twist compatible coresolution. The proof of (iii) is similar (one must work in the double reduced bar resolution $\overline{B}(\overline{B}(S^*))$, as in \cite[Prop. 3.21]{TouzeUniv}). Finally, (iv) follows from the stability of $\mathcal{T}_r\mathcal{P}_\k$ by tensor products.
\end{proof}

We are now ready to prove the main theorem of section \ref{sec-collapse}.

\begin{theorem}[{A partial answer to conjecture \ref{conj-collapse}}]\label{thm-collapse}
Let $F,G\in\P_\k$ and let $r$ be a positive integer. Assume that $G$ admit an $r$-twist compatible coresolution. Then there is a graded isomorphism, natural in $F$ (take the total degree on the right handside):
$$\Ext^*(F^{(r)},G^{(r)})\simeq \Ext^*(F,G(E_r\otimes I))\;. $$
Thus, the twisting spectral sequence $E(F,G,r)$ collapses at the second page.

Similarly, if $F^\sharp$ has an $r$-twist compatible coresolution, $\Ext^*(F^{(r)},G^{(r)})$ is graded isomorphic (naturally in $G$) to $\Ext^*(F(E_r\otimes I),G)$, and the twisting spectral sequence $E(F,G,r)$ collapses at the second page.
\end{theorem}

\begin{remark}
Theorem \ref{thm-collapse} generalizes \cite[Cor 4.2]{Chalupnik2} (One can prove that Cha{\l}upnik's result is equivalent to the case $F^\sharp=\Lambda^d$). The latter result is the key result (together with \cite[proposition 3.1]{Chalupnik2}) to compute the extension groups $\Ext^*(\Lambda^{*\,(j)},\Gamma^{*\,(r)})$ and $\Ext^*(S^{*\,(j)},\Gamma^{*\,(r)})$.
\end{remark}

\begin{proof}[Proof of theorem \ref{thm-collapse}]
Assume that $G$ admits a $r$-twist compatible coresolution $J^*$. Then the totalization of the bicomplex $T(J^*,r)^*$ is an injective coresolution of $G^{(r)}$. Hence $\Ext^*(F^{(r)},G^{(r)})$ may be computed as the homology of the totalization of the bicomplex $B^{*,*}(F,r):=\hom(F^{(r)},T(J^*,r)^*)$

Now by proposition \ref{prop-analysis}, the odd rows $B^{*,2k+1}(F,r)$ of $B^{*,*}(F,r)$ are zero (so the vertical differential of this bicomplex is trivial) and the totalization of $B^{*,*}(F,r)$ is isomorphic (naturally in $F$) to $\hom(F,J^*(E_r\otimes I))$. 

But the homology of the latter complex computes $\Ext^*(F,G(E_r\otimes I))$, whence the first isomorphism.
 
The collapsing of the twisting spectral sequence follows for dimension reasons by corollary \ref{cor-num}(ii). Finally, the case when $F^\sharp$ has an $r$-twist injective coresolution follows by duality \cite[Prop 2.6]{FS}.
\end{proof}

\subsection{A combinatorial problem}\label{subsec-comb}

Let $(C^*,\partial)$ be a complex in $\P_\k$, whose objects are finite direct sums of functors of the form $S^\mu$. For each object $C^k:=\bigoplus_{i=1}^n S^{\mu^i}$, we define an object $\widetilde{C}^k:=\bigoplus_{i=1}^n S^{p\mu^i}$. The canonical inclusions $S^{\mu\,(1)}\hookrightarrow S^{p\mu}$ yield a graded injection: $C^{*\,(1)}\hookrightarrow \widetilde{C}^*$.  
Since the objects of $\widetilde{C}^*$ are injective it is always possible to find maps $\widetilde{\partial}^k:\widetilde{C}^k\to \widetilde{C}^{k+1}$ whose restrictions to $C^{k\,(1)}$ equal $\partial^{k\,(1)}$. But if we construct the $\widetilde{\partial}^k$ by an abstract injectivity argument, we do not have $\widetilde{\partial}\circ \widetilde{\partial}=0$ in general.

\begin{problem}\label{pb-lift}
For all $F\in\P_\k$, find a coresolution $(J^*_F,\partial)$ by direct sums of functors of the form $S^\mu$ such that 
the following holds. There exists a family of liftings $\widetilde{\partial}^k:\widetilde{J}^k_F\to \widetilde{J}^{k+1}_F$, $k\ge 0$, of the differentials $\partial^{k\,(1)}$, such that $\widetilde{\partial}\circ \widetilde{\partial}=0 $. 
\end{problem}

If $F$ admits a $1$-twist compatible coresolution $(J^*_F,\partial)$, then $\overline{\partial}$ is a lifting of $\partial^{(1)}$ to $\widetilde{J}_F^*$ such that $\overline{\partial}\circ\overline{\partial}=0$, so in that case we can solve the problem.

An answer to the problem \ref{pb-lift} would imply a positive answer to the collapsing conjecture \ref{conj-collapse}.
Indeed, assume that we can solve the problem for $G$. Then, by propositions \ref{prop-res1} and \ref{prop-res-2}, the bigraded object $J_G^*(\sha_r\otimes I)$ is endowed  with a vertical $p$-differential $\delta$ and  a horizontal differential $\overline{\partial}(\sha_r\otimes I)$ which commute. Moreover, if we contract the columns (i.e. apply the functor $-_{[1]}$ columnwise), we obtain a bicomplex $\widetilde{J}_G^{*,*}$ whose totalization is an injective coresolution of $F$. Then for all $F\in\P_\k$ we can analyze $\hom(F^{(r)}, \widetilde{J}_G^{*,*})$ as in the proof of theorem \ref{thm-collapse} to conclude that $E(F,G,1)$ collapses. Now, as observed in remark \ref{rk-collapse}, if $E(F,G,1)$ collapses for all $F,G$, then conjecture \ref{conj-collapse} is valid.

\section{Appendix: review of the construction of Troesch $p$-complexes}\label{sec-app}

In this appendix, we review the construction of Troesch $p$-complexes from \cite{Troesch} (with our notations and conventions from section \ref{sec-background}). As usual, we work over a ground field $\k$ of prime characteristic $p$.

\subsection{A preparatory lemma}
Let $U$ be a finite dimensional graded vector space over $\k$. We consider $S^*(U)$ as a graded object in the usual way (e.g. if $x$ (resp. $y$) in a homogeneous element of degree $i$ (resp. $j$), then $xy\in S^*(U)$ has degree $i+j$). Recall that if $C$ is a coalgebra, $A$ is an algebra, and $f,g\in\End_\k(C,A)$, the convolution $f\star g:C\to A$ is the composite:
$$C\xrightarrow[]{\Delta}C\otimes C\xrightarrow[]{f\otimes g}A\otimes A\xrightarrow[]{m}A\;.$$
The following lemma gives a way to construct $p$-differentials on $S^*(U)$ from $p$-differentials on $U$. It encompasses the computations of \cite[Prop 4.1.4, 4.1.5 and 4.2.6]{Troesch}. 
\begin{lemma}\label{lm-diff}
Let $\phi:U\to U$ be a morphism of graded vector spaces, which raises the degree by an integer $\alpha$. The convolution 
$\Id\star S^*(\phi) $ splits as a sum of morphisms 
$\phi_d: S^*(U)\to S^*(U) $ such that for all $d\ge 0$  $\phi_d$ raises the degree by $d\alpha$. The morphisms $\phi_d$ satisfy the following properties.
\begin{enumerate}
\item[(1)] For all $d,e\ge 0$, $\phi_d\circ \phi_e=\phi_e\circ\phi_d$.
\item[(2)] If $\phi$ is a $p$-differential, then for all $d\ge 1$, $\phi_d$ is a $p$-differential.
\item[(3)] Let $\psi:U\to U$ be another morphism of graded vector spaces, which raises the degree by $\beta$. If $\phi$ and $\psi$ commute, then for all $d,e\ge 0$, $\phi_d\circ\psi_e=\psi_e\circ\phi_d$.
\end{enumerate}
\end{lemma}
\begin{proof}
Observe that (3) implies (1). Let us prove (3).
For all $d\ge 0$, the restriction of $\phi_d$ to the graded object $S^n(U)$ can be explicitly described as the composite
$$S^n(U)\xrightarrow[]{\Delta} S^{n-d}(U)\otimes S^d(U)\xrightarrow[]{\Id\otimes S^d(\phi)}S^{n-d}(U)\otimes S^d(U)\xrightarrow[]{m} S^n(U)\;.$$
If $(x_1,\dots ,x_n)$ is a family of elements of $U$ and if $J\subset \{1,\dots,n\}$, we denote by $x_J$ the product of the $x_j$, for $j\in J$ (so $x_J\in S^{\mathrm{card}(J)}(U)$). The comultiplication $\Delta: S^n(U)\to S^{n-d}(U)\otimes S^d(U)$ sends a product $x_1\cdots x_n\in S^n(U)$ to the sum of the $x_J\otimes x_K\in S^{n-d}(U)\otimes S^d(U)$, where the sum is taken over all partitions $J\sqcup K$ of $\{1,\dots,n\}$ into a subset $J$ of $n-d$ elements and a subset $K$ into $d$ elements. So we have the equality: $$\phi_d(x_1\cdots x_n)=\sum_{J\sqcup K=\{1,\dots,n\}\; \mathrm{card}(K)=d} x_J\cdot\phi(x)_K \;.$$
A similar formula holds for $\psi_e$:
$$\psi_e(y_1\cdots y_n)=\sum_{L\sqcup M=\{1,\dots,n\}\; \mathrm{card}(M)=e} y_L\cdot\phi(y)_M \;.$$
Putting these two formulas together, we compute that the composite $\psi_e\circ\phi_d$ is given by the formula:
$$(\psi_e\circ \phi_d)(x_1\cdots x_n)=\sum x_{J\cap L}\cdot \phi(x)_{K\cap L}\cdot\psi(x)_{J\cap M}\cdot(\psi\circ\phi)(x)_{K\cap M}\;,$$
where the sum is taken over all partitions $J\sqcup K=\{1,\dots,n\}$ and $L\sqcup M=\{1,\dots,n\}$ such that $K$ has cardinal $d$ and $M$ has cardinal $e$. Similarly, 
$$(\phi_d\circ \psi_e)(x_1\cdots x_n)=\sum x_{J\cap L}\cdot \phi(x)_{K\cap L}\cdot\psi(x)_{J\cap M}\cdot(\phi\circ\psi)(x)_{K\cap M}\;.$$
Since $\phi\circ\psi=\psi\circ\phi$ the two expressions are equal. This proves (3).

Now we prove (2). Since the $\phi_d$ commute with each other by (1), Newton's binomial formula yields an equality:
$$(\Id\star S^*(\phi))^p=\sum_{d=0}^\infty (\phi_d)^p\;.$$
In this formula, the sum is finite on each summand $S^n(U)$, and the exponent refers to composition of morphisms. So, to prove that the iterated compositions $(\phi_d)^p$ are zero for $d> 0$, it suffices to prove that $(\Id\star S^*(\phi))^p$ equals the identity map. Let $f,g:S^*(U)\to S^*(U)$ be two Hopf algebra morphisms. Then it follows from the axioms of commutative Hopf algebras satisfied by $S^*(U)$ that: 
$$(\Id\star f)\circ (\Id\star g)= \Id\star (f\star g\star (f\circ g))\;.$$
Using the axioms of Hopf algebras, and the commutativity and the cocommutativity of $S^*(U)$, we check that $(f\star g\star (f\circ g))$ is a morphism of Hopf algebras. So we can iterate the computation. In particular, if $f$ is of the form $S^*(\phi)$, with $\phi$ an endomorphism of $U$, we get the formula:
$$(\Id\star S^*(\phi))^p= \Id\star(S^*(\phi)^p)\;.$$
If $\phi$ is a $p$-differential, then $S^*(\phi)^p=S^*(\phi^p)=S^*(0)=\eta\circ\epsilon$ (where $\eta$, resp $\epsilon$ is the unit, resp. counit, of $S^*(U)$). Since $\eta\circ\epsilon$ is the unit for the convolution, we obtain that $(\Id\star S^*(\phi))^p=\Id$. This concludes the proof of lemma \ref{lm-diff}.
\end{proof}

\subsection{Construction of Troesch $p$-complexes}
We now turn to describe the construction of the $p$-complexes $B_d(r)^*$. Let $V$ be a finite dimensional vector space (the elements of $V$ are considered as elements of degree $0$). As a graded object, we let:
$$B_d(r)^*(V):=S^*(\sha_r\otimes V)\;,$$
where $\sha_r$ is the graded vector space which equals $\k$ in each degree $i$, for $0\le i<p^r$. 
To define the $p$-differential on $B_d(r)^*(V)$, we proceed in several steps. 

{\bf Step 1.} First, we observe that there is a canonical isomorphism of graded vector spaces:
$$\sha_1^{(0)}\otimes \sha_1^{(1)}\otimes\cdots\otimes\sha_1^{(r-1)}\simeq \sha_r\;.$$
Indeed each $\sha_1^{(k)}$ is (according to the conventions of section \ref{sec-background}, which define the Frobenius twist $I^{(k)}$ on graded vector spaces\footnote{Here, we use the convention $\sha_1^{(0)}=\sha_1$.} as an operator which multiplies the degrees by $p^k$) a direct sum of copies of $\k$, one in each degree $p^ki$ for $0\le i<p$. 

{\bf Step 2.} On each $\sha_1^{(k)}$ we define a `translation operator' $\rho^{(k)}:\sha_1^{(k)}\to \sha_1^{(k)}$, which maps the summand $\k$ of degree $p^ki$ of $\sha_1^{(k)}$ identically onto the summand $\k$ of degree $p^k(i+1)$ of $\sha_1^{(k)}$ if $i<p-1$, and which is zero in the other degrees. Thus, $\rho^{(k)}$ raises the degree by $p^k$ and it is a $p$-differential: $(\rho^{(k)})^p=0$.

{\bf Step 3.} Now we use the translation operators $\rho_k$ and lemma \ref{lm-diff} to define a bunch of $p$-differentials on the graded functor $B_d(r)^*(V)$. Namely, for all integer $s$ such that $0\le s< r$, and for all positive integer $\ell$ we let  
$$ d^{r-1-s}_\ell := \left(\sha_1^{(0)}\otimes\cdots\otimes \sha_1^{(s-1)}\otimes\rho^{(s)}\otimes \sha_1^{(s+1)}\otimes \cdots\otimes \sha_1^{(r-1)}\otimes V\right)_{\ell}.$$
That is, $d_\ell^{r-1-s}$ is the part of $\Id\star S^*(\sha_1^{(0)}\otimes\cdots \otimes \rho^{(s)}\otimes\cdots\otimes\sha_1^{(r-1)}\otimes V )$ which raises the degree by $\ell p^s$. By lemma \ref{lm-diff}, these morphisms $d_\ell^s$ satisfy the following properties.
\begin{enumerate}
\item[(i)] For all $\ell> 0$ and all $s\in\{0,\dots,r-1\}$, $d_\ell^s$ raises the degree by $\ell p^{r-1-s}$.
\item[(ii)] For all $\ell> 0$ and all $s\in\{0,\dots,r-1\}$, $d_\ell^s$ is a $p$-differential.
\item[(iii)] For all $k,\ell> 0$ and all $s,t\in\{0,\dots,r-1\}$, $d_\ell^s\circ d_k^t= d_k^t\circ d_\ell^s$.
\end{enumerate}

{\bf Step 4.} Finally, we define the differential $d$ on $B_d(r)^*$, which raises the degree by $p^{r-1}$, by the formula:
$$d:= d_1^0+d_p^1+\dots+d_{p^{r-1}}^{r-1}\;. $$
(It is a $p$-differential thanks to properties (ii) and (iii) given in step 3.)

\bigskip

\begin{remark}
If $r=1$, the $p$-complex $B_d(1)^*(V)$ equals $S^d(\sha_1\otimes V)$ as a graded object. Its differential $d=d^0_1$ equals the composite:
\begin{align*}S^d(\sha_1\otimes V)\xrightarrow[]{\Delta}  & S^{d-1}(\sha_1\otimes V)\otimes (\sha_1\otimes V)\\ &\xrightarrow[]{\Id\otimes( \rho\otimes V)} S^{d-1}(\sha_1\otimes V)\otimes (\sha_1\otimes V)\xrightarrow[]{m}S^d(\sha_1\otimes V)\;.\end{align*}
The exponential formula yields an isomorphism of graded functors $S^d(\sha_1\otimes V)\simeq \bigoplus S^\mu(V)$, where the sum is taken over all $p$-tuples $\mu=(\mu_0,\dots,\mu_{p-1})$ of weight $d$, and each $S^\mu(V)$ is placed in degree $\sum i\mu_i$. If we denote by $\delta:S^i\otimes S^j\to S^{i-1}\otimes S^{j+1}$ the composite
$$S^i\otimes S^j\xrightarrow[]{\Delta}S^{i-1}\otimes S^1\otimes S^{j}\xrightarrow[]{m} S^{i-1}\otimes S^{j+1}\;, $$
Then the differential $d$ can be rewritten as the sum
$$d=\sum_{i=0}^{p-1}\underbrace{\Id\otimes\cdots\otimes\Id}_{i\text{ factors}}\otimes \delta \otimes \underbrace{\Id\otimes\cdots\otimes\Id}_{p-2-i\text{ factors}} $$
This corresponds to the expression given in \cite[Section 3.1, p. 1062]{Troesch}.
\end{remark}

\begin{remark}
If $r\ge 2$, we denote by $\sha_1^{s,r}$ the graded vector space:
$$\sha_1^{s,r}:=\sha_1^{(s)}\otimes \sha_1^{(0)}\otimes \cdots\otimes \widehat{\sha_1^{(s)}}\otimes \cdots\otimes \sha_1^{(r-1)}\;.$$
Permutation of the factors of tensor products yields isomorphisms $\chi_s:\sha_1^{s,r}\simeq \sha_1^{0,s}$ for all $0\le s<r$. The maps $\iota_s(r)$ defined in \cite[Def 4.2.1]{Troesch} can be identified, though the exponential isomorphism, with the isomorphisms $S^*(\sha_1^{s,r}\otimes V)\xrightarrow[]{S^*(\chi_s)}S^*(\sha_1^{0,r}\otimes V)$. 
Thus, the differentials $d^s_\ell$ defined by Troesch in \cite[Def 4.2.2]{Troesch} coincide with ours. 
\end{remark}

\subsection{The homology of $B_d(1)^*$} The following result is proved in \cite[Section 3]{Troesch}. We quote it without proof.

\begin{theorem}[{\cite[Thm 3.1.2]{Troesch}}]\label{thm-calcul-Troesch}
Let $d\ge 0$. let us identify $S^{d\,(1)}$ with a subfunctor of $B_{pd}(1)^*$ of cohomological degree $0$ via the inclusion
$$S^{d\,(1)}(V)\hookrightarrow S^{pd}(V) \hookrightarrow S^{pd}(\sha_1\otimes V)\;.$$
Then the $p$-complex $B_{pd}(1)^*$ is a $p$-coresolution of $S^{d\,(1)}$ (in particular, the $p$-differential of $B_{pd}(1)^*$ vanishes on $S^{d\,(1)}$).

If $d$ is not divisible by $p$, then $B_{d}(1)^*$ is $p$-acyclic.
\end{theorem}

\subsection{Some maps between Troesch $p$-complexes}\label{subsec-iota}
Let us define a morphism:
$$\iota:B_d(r)^*(V^{(1)})=S^d(\sha_r\otimes V^{(1)})\to S^{pd}(\sha_{r+1}\otimes V)=B_{pd}(r+1)^*(V)$$  
as the composite:
$$S^{d}(\sha_{r}\otimes V^{(1)})\simeq S^{d\,(1)}(\sha_{r}\otimes V) \hookrightarrow S^{pd}(\sha_{r}\otimes V) \hookrightarrow S^{pd}(\sha_{r+1}\otimes V) \;.$$
Here, the first morphism is induced by the isomorphism $\sha_{r}\simeq \sha_{r}^{(1)}$ which maps the summand $\k$ of degree $i$ of  $\sha_{r}$ identically onto the summand $\k$ of degree $pi$ of $\sha_{r}^{(1)}$. The second map is induced by the canonical inclusion $S^{d\,(1)}\hookrightarrow S^{pd}$, and the last map is induced by the canonical inclusion of graded objects of $\sha_{r}$ into $\sha_{r}\otimes \sha_{1}^{(r)}=\sha_{r+1}$.
The following lemma follows from the construction of the differentials of $B_d(r)^*$. We use it in lemma \ref{lm-prep2}.
\begin{lemma}[{Compare \cite[Lemmas 4.3.4 and 4.3.5]{Troesch}}]\label{lm-fonct-B}
The morphism $\iota$ sends an element of degree $i$ of $B_d(r)^*(V^{(1)})$ to an element of degree $pi$ of $B_{pd}(r+1)^*(V)$. Moreover, it commutes with the $p$-differentials.
\end{lemma}
\begin{proof}
The assertion on degrees is clear. Let us prove that $\iota$ commutes with the $p$-differentials. In this proof, we denote by $d$ the $p$-differential of $B_d(r)^*$, and by $D$ the $p$-differential of $B_{pd}(r+1)^*$. We also denote by $\iota'$ the inclusion of graded functors:
$$ S^{d\,(1)}(\sha_{r}\otimes V) \hookrightarrow S^{pd}(\sha_{r}\otimes V) \hookrightarrow S^{pd}(\sha_{r+1}\otimes V)\;. $$

{\bf Step 1.}
We define a $p$-differential $d'$ on $S^d(\sha_r^{(1)}\otimes V^{(1)})=S^{d\,(1)}(\sha_r\otimes V)$, which raises the degree by $p^{r}$, as the sum
$$d':= {d'}_p^1+\dots+{d'}_{p^{r}}^{r}\;, $$
where ${d'}_\ell^s$ denotes the part $\Id\star S^d(\sha_1^{(1)}\otimes\cdots \otimes \rho^{(s)}\otimes\cdots\otimes\sha_1^{(r)}\otimes V^{(1)} )$ which raises the degree by $\ell p^{r-s}$. 

By definition of $d'$, the isomorphism 
$S^{d}(\sha_{r}\otimes I^{(1)})\simeq S^{d\,(1)}(\sha_{r}\otimes I)$ induces an isomorphism of $p$-complexes $(B_d(r)^*(V^{(1)}),d^{(1)})\simeq (S^{d\,(1)}(\sha_r\otimes V),d')$, which sends an element of degree $i$ of the first $p$-complex to an element of degree $pi$ of the second $p$-complex.

Thus, to conclude the proof of lemma \ref{lm-fonct-B}, it suffices to prove that $\iota'$ induces a degree preserving morphism of $p$-complexes:
$$(S^{d\,(1)}(\sha_r\otimes V),d')\hookrightarrow (B_{pd}(r+1)^*(V),D) $$

{\bf Step 2.}
We write the $p$-differential $D$ of $B_{pd}(r+1)^*(V)$ as the sum of two $p$-differentials: 
$$D=D'+D''\;\;\text{ with }\;\; D'=D_1^0+\cdots + D_{p^{r-1}}^{r-1}\;\;\text{ and }\;\;D''=D_{p^{r}}^{r}\;,$$ 
where $D_\ell^s$ is the part of $\Id\star S^{pd}(\sha_1^{(0)}\otimes\cdots \otimes \rho^{(s)}\otimes\cdots\otimes\sha_1^{(r)}\otimes V )$ which raises the degree by $\ell p^{r-s}$. By definition of $d'$ and $D'$, the inclusion $\iota':S^{d\,(1)}(\sha_{r}\otimes I) \hookrightarrow 
S^{pd}(\sha_{r}\otimes\sha_1^{(r)}\otimes V)$ induces an inclusion of $p$-complexes 
$$(S^{d\,(1)}(\sha_r\otimes V),d')\hookrightarrow (B_{pd}(r+1)^*(V),D')\;.$$

{\bf Step 3.} So, to finish the proof, it suffices to check that $D''$ vanishes on the graded subfunctor $S^{d\,(1)}(\sha_r\otimes V)$ of $B_{pd}(r+1)^*(V)$. 

We have a commutative diagram:
$$\xymatrix{
S^{pd}(\sha_1\otimes\sha_r\otimes V)\ar[r]^{\simeq}_{(*)}& S^{pd}(\sha_r\otimes\sha_1^{(r)}\otimes V)\\
S^{d\,(1)}(\sha_r\otimes V)\ar@{^{(}->}[u]\ar@{=}[r]&S^{d\,(1)}(\sha_r\otimes V)\ar@{^{(}->}[u]_{\iota'}
}$$
where the vertical arrow on the left is the inclusion from theorem \ref{thm-calcul-Troesch} evaluated on $\sha_r\otimes V$, the vertical arrow on the right is our map $\iota'$ and the horizontal arrow $(*)$ is induced by the isomorphism $\sha_1\simeq \sha_1^{(r)}$ which maps an element of degree $i$ of $\sha_1$ identically on same the same element with degree $p^ri$ of $\sha_1^{(r)}$.

Let us denote by $d''$ the $p$-differential of $B_d(1)^*$. By definition of $d''$ and $D''$, the isomorphism $(*)$ induces an isomorphism of $p$-complexes
$$(B_{pd}(1)^*(\sha_r\otimes V),d''(\sha_r\otimes V))\simeq (B_{pd}(r+1)^*(V),D'')\;.$$

We know by theorem \ref{thm-calcul-Troesch} that the differential $d''(\sha_r\otimes V)$ vanishes on $S^{d\,(1)}(\sha_r\otimes V)$. Since the diagram commutes, this means that $D''$ vanishes on $S^{d\,(1)}(\sha_r\otimes V)$. This concludes the proof.
\end{proof}

\subsection{The homology of $B_d(r)^*$} The following result corresponds to \cite[Thm 4.3.2]{Troesch}. 
For the sake of completeness, we give a sketch of Troesch's proof below.
\begin{theorem}[{\cite[Thm 4.3.2]{Troesch}}]\label{thm-calcul-Troesch2}
Let $r$ be a positive integer and let $d$ be a nonnegative integer.
The $p$-complex $B_{p^rd}(r)^*$ is a $p$-coresolution of $S^{d\,(r)}$.
If $d$ is not a multiple of $p^r$ then the $p$-complex $B_{d}(r)^*$ is $p$-acyclic.
\end{theorem}
\begin{proof}[Proof (Troesch)]
The proof goes by induction on $r$. For $r=1$, the result is given by theorem \ref{thm-calcul-Troesch}. So we now assume that theorem \ref{thm-calcul-Troesch2} is true for $r$, and we want to prove it for $r+1$. We concentrate on the case of the homology of $B_{dp^{r+1}}(r+1)^*$ (the case of the homology of $B_{d}(r+1)^*$ when $p^r\not|d$ is similar).

We denote by $d$ the $p$-differential $B_{p^rd}(r)^*$, and by $D$ the differential of $B_{p^{r+1}d}(r+1)^*$.  We have seen in the proof of lemma \ref{lm-fonct-B} that $D$ can be written as the sum of two commuting $p$-differentials $D'$ and $D''$ such that the following two conditions are satisfied.
\begin{itemize}
\item[(a)] The inclusion $\iota: B_{p^rd}(r)^{*\,(1)}\hookrightarrow B_{p^{r+1}d}(r+1)^*$ from section \ref{subsec-iota} induces a morphism of $p$-complexes:
$$(B_{p^rd}(r)^{*\,(1)}, d^{(1)})\hookrightarrow (B_{p^{r+1}d}(r+1)^*,D')$$
which sends an element of degree $i$ of $B_{p^rd}(r)^{*\,(1)}$ to an element of degree $pi$ of $B_{p^{r+1}d}(r+1)^*$ (this is proved in the first two steps of the proof of  lemma \ref{lm-fonct-B}).
\item[(b)] $(B_{p^{r+1}d}(r+1)^*,D'')$ is a $p$-coresolution of $B_{p^rd}(r)^{*\,(1)}$ (this is a consequence of step 3 of the proof of  lemma \ref{lm-fonct-B}).
\end{itemize}

If $p=2$, then we can say that $(B_{p^{r+1}d}(r+1)^*,D',D'')$ is an ordinary bicomplex, and there is an associated spectral sequence converging to the homology of the total complex, that is $(B_{p^{r+1}d}(r+1)^*,D)$. The zero-th page of the spectral sequence is the complex  $(B_{p^{r+1}d}(r+1)^*,D'')$, whose homology equals $B_{p^rd}(r)^{*\,(1)}$, concentrated in the first column by condition $(b)$. Now condition $(a)$ ensures that the first page of the spectral sequence equals the complex $(B_{p^rd}(r)^{*\,(1)}, d^{(1)})$ concentrated in the first column. By induction hypothesis, its homology equals $S^{d\,(r+1)}$, concentrated in bidegree $(0,0)$. Hence $(B_{p^{r+1}d}(r+1)^*,D)$ is a coresolution of $S^{d\,(r+1)}$ and we are done.

If $p>2$, the theory of spectral sequences arising from $p$-bicomplexes does not exist, but Troesch shows in \cite[Lemma 4.3.6]{Troesch} that the conclusion remains valid for bicomplexes of that kind. Hence we obtain that $(B_{p^{r+1}d}(r+1)^*,D)$ is a $p$-coresolution of $S^{d\,(r+1)}$.
\end{proof}

\end{document}